\documentclass[final,leqno,onefignum,onetabnum]{siamart1116}

\usepackage{amssymb}
\usepackage{stmaryrd}
\usepackage{mathtools}
\usepackage{tikz}
\usepackage{pgfplots}
\usepackage{pgfplotstable}
\usepackage[caption=false]{subfig}

 \DeclarePairedDelimiter{\parens}{\lparen}{\rparen}
\DeclarePairedDelimiter{\set}{\{}{\}}

\DeclareMathOperator*{\argminB}{argmin}

\hypersetup{colorlinks=false}

\newcommand{\A}[0]{\mathcal{A}}

\newcommand{\Eps}[0]{\boldsymbol{\varepsilon}}
\newcommand{\Kappa}[0]{\boldsymbol{K}}
\newcommand{\ConstRel}[0]{\mathbb{C}\,}
\newcommand{\Mom}[0]{\boldsymbol{M}}
\newcommand{\Shear}[0]{\boldsymbol{Q}}
\newcommand{\Div}[0]{\mathrm{div}\,}
\newcommand{\Gf}[0]{\Gamma_\mathrm{f}}
\newcommand{\Gs}[0]{\Gamma_\mathrm{s}}
\newcommand{\Gc}[0]{\Gamma_\mathrm{c}}
\newcommand{\Vn}[0]{V_{n}}
\newcommand{\mnn}[0]{M_{nn}}
\newcommand{\qn}[0]{Q_n}
\newcommand{\effs}[0]{V_n}

\newcommand{\dmnsuh}{\frac{\partial M_{ns}(u_h)}{\partial s}}

\newcommand{\mns}[0]{M_{ns}}
\newcommand{\msn}[0]{M_{sn}}

\newcommand{\smooth}[0]{{\mathcal S}}

\newcommand{\osc}[0]{\mathrm{osc}}

\newcommand{\Ch}[0]{\mathcal{C}_h}
\newcommand{\Eh}[0]{\mathcal{E}_h}
\newcommand{\Ehint}[0]{\mathcal{E}_h^I}
\newcommand{\Ehsimply}[0]{\mathcal{E}_h^S}
\newcommand{\Ehfree}[0]{\mathcal{E}_h^F}
\newcommand{\sub}{R}

\newcommand{\vertiii}[1]{{\left\vert\kern-0.25ex\left\vert\kern-0.25ex\left\vert #1 
            \right\vert\kern-0.25ex\right\vert\kern-0.25ex\right\vert}}

\theoremstyle{definition}
\newtheorem{prob}{Problem}
\newtheorem{defn}{Definition}
\newtheorem{alg}{Algorithm}

\numberwithin{equation}{section}

\title{A posteriori estimates for conforming  Kirchhoff plate elements\thanks{Funding from Tekes -- the Finnish Funding Agency for Innovation (Decision number 3305/31/2015) and the Finnish Cultural
    Foundation is gratefully acknowledged, as as well as the financial support from FCT/Portugal through UID/MAT/04459/2013.  
}}
\author{Tom Gustafsson\thanks{Department of Mathematics and Systems Analysis,
Aalto University, P.O. Box 11100,  \hfill \break00076 Aalto, Finland
e-mail: 
(\email{tom.gustafsson@aalto.fi}). }  \and Rolf Stenberg\thanks{Department of Mathematics and Systems Analysis,
Aalto University, P.O. Box 11100,  \hfill \break00076 Aalto, Finland
e-mail: 
(\email{rolf.stenberg@aalto.fi}). }  
 \and {Juha Videman}\thanks{CAMGSD/Departamento de Matem\'atica, Instituto Superior T\'ecnico, Universidade   de Lisboa, Av. Rovisco Pais 1, 1049-001 Lisboa, Portugal (\email{jvideman@math.tecnico.ulisboa.pt}). }  }
\begin{document}

\maketitle

\begin{abstract}
 We derive a residual a posteriori estimator for the Kirchhoff plate bending problem. We consider the problem with a combination of clamped, simply supported and free boundary conditions subject to both distributed and concentrated (point and line) loads. Extensive numerical computations are presented to verify the functionality of the estimators.
\end{abstract}



\section{Introduction}
The purpose of this paper is to perform an a posteriori error analysis of conforming finite element methods for the classical Kirchhoff plate bending model. So far this has not been done in full generality as it comes to the boundary conditions. Most papers deal only with clamped or simply supported boundaries, see \cite{VerfurthI} for conforming $C^1$ elements,  \cite{Charbonneau,Gudi2011,VerfurthI} for the mixed Ciarlet--Raviart method (\cite{Ciarlet-Raviart}), and \cite{ Brenner-etal2010,Brenner2010,Hansbo-Larson,Georgoulis2011,MR3267358,MR3261327,Huang2016} for discontinuous Galerkin (dG) methods. The few papers that do address more general boundary conditions, in particular free, are \cite{BdVNSMG,Hu-Shi} in which the nonconforming Morley element is analysed, \cite{BdVNS-I,BdVNS-II} where a new mixed method is introduced and analysed, and \cite{Hansbo-Larson} where a continuous/discontinuous Galerkin method is considered. One should also note that the Ciarlet--Raviart method cannot even be defined for general boundary conditions. Free boundary conditions could be treated using dG methods following an analysis similar to the one presented here. 

In this study, we will derive a posteriori estimates using conforming methods and allowing for a combination of clamped, simply supported and free boundaries. In addition, we will investigate the effect of concentrated point and line loads, which are not only admissible in our $H^2$-conforming  setting but of great engineering interest, on our a posteriori bounds in numerical experiments. We note that finite element approximation of elliptic problems with loads acting on lower-dimensional manifolds has been considered by optimal control theory, see \cite{Gong2014} and all the references therein.

The outline of the paper is the following. In Section 2, we recall the Kirchhoff-Love plate model by presenting its variational formulation and the corresponding boundary value problem. We perform this in detail for the following reasons. First, as noted above, general boundary conditions are rarely considered in the numerical analysis literature. 
Second, the free boundary conditions consist of a vanishing normal moment and a vanishing Kirchhoff shear force. 
These arise from the variational formulation via successive integrations by parts. It turns out that the same 
 steps are needed in the a posteriori analysis in order to obtain a sharp estimate, i.e. both reliable and efficient. 
In the following two sections, we present the classical conforming finite element methods and derive new a posteriori error estimates. In the last section, we present the results of our numerical experiments computed with the triangular Argyris element. We consider the point, line and square load cases with simply supported boundary conditions in a square domain as well as solve the problem in an L-shaped domain with uniform loading  using different combinations of boundary conditions.

\section{The Kirchhoff plate model}

The dual kinematic and force variables in the model are the curvature and the
moment tensors. Given the deflection $u$ of the midsurface of the plate, the
curvature is defined through 
\begin{equation}\label{curvdef}
  \Kappa(u) = -\Eps(\nabla u), \end{equation}
with the infinitesimal strain operator defined by
\begin{equation}
    \Eps(\boldsymbol{v})
    = \frac12 \parens*{\nabla \boldsymbol{v}+\nabla \boldsymbol{v}^T},
\end{equation}
where $(\nabla \boldsymbol{v})_{ij} = \frac{\partial v_i}{\partial x_j}$.
The dual force variable, the moment tensor $\Mom$,  is related to
$\Kappa$ through the  constitutive relation
\begin{equation}
  \qquad \Mom(u) = \frac{d^3}{12} \ConstRel \Kappa(u),
\end{equation}
where $d$ denotes the plate thickness and where we have assumed an isotropic linearly elastic material, i.e.
\begin{equation}
    \label{eq:constrel}
    \ConstRel \boldsymbol{A}
    = \frac{E}{1 + \nu}
      \parens*{\boldsymbol{A}
               + \frac{\nu}{1 - \nu}(\text{tr}\,\boldsymbol{A}) \boldsymbol{I}},
      \quad \forall \boldsymbol{A} \in \mathbb{R}^{2 \times 2}.
\end{equation}
Here $E$ and $\nu$ are the Young's modulus and Poisson ratio, respectively. 
The shear force is denoted by
$\Shear=\Shear(u)$. The moment equilibrium equation reads as 
\begin{equation}
     \boldsymbol{\Div} \Mom(u) = \Shear(u),
\end{equation}
where $   \boldsymbol{\Div}$ is the vector-valued divergence operator applied to tensors.
The transverse shear equilibrium equation is
\begin{equation}
 -\Div \Shear(u) = l,
\end{equation}
with  $l$ denoting the transverse loading. Using the constitutive
relationship~\eqref{eq:constrel}, a straightforward elimination yields the well-known Kirchhoff--Love plate equation:
\begin{equation}
    \A(u) := D \Delta^2 u = l,
\end{equation}
where the so-called bending stiffness $D$ is defined as
\begin{equation}
    D = \frac{E d^3}{12(1-\nu^2)}.
\end{equation}

Let $\Omega \subset \mathbb{R}^2$ be a  polygonal domain that describes
the midsurface of the plate.  The plate is considered to be clamped on $\Gc
\subset \partial \Omega$, simply supported on $\Gs \subset \partial \Omega$ and
free on $\Gf \subset \partial \Omega$ as depicted in Fig.~\ref{fig:plate}. The
loading is assumed to consist of a distributed load $f\in L^2(\Omega)$, a load
$g\in L^2(S)$ along the line $S\subset \Omega$, and a point load $F$ at an interior 
point $x_0 \in \Omega$.

Next, we will  turn to the boundary conditions, which are best understood from
the variational formulation. (Historically, this was also how they were first
discovered by Kirchhoff, cf.~\cite{T-hist}.) The elastic energy of the plate as a function of the
deflection $v$ is $\frac{1}{2} a(v,v)$, with the bilinear form $a$ defined by
\begin{equation}
    a(w,v) = \int_\Omega \Mom(w) : \Kappa(v) \,\mathrm{d}x = \int_\Omega \frac{d^3}{12} \ConstRel \Eps(\nabla w) : \Eps(\nabla v)\,\mathrm{d}x,
\end{equation}
and the potential energy due to the loading is
\begin{equation}\label{ldef}
l(v)= \int_\Omega f v \,\mathrm{d}x + \int_S g v \, \mathrm{d} s + F v(x_0).
\end{equation}
Defining the space of kinematically admissible deflections
\begin{equation}
V=\{ \, v \in H^2(\Omega) :  v\vert_{\Gc\cup \Gs}=0, \, \nabla v \cdot \boldsymbol{n}\vert_{\Gc}=0      \},
\end{equation} 
minimization of the total energy 
\begin{equation} 
u=\argminB_{v\in V} \Big\{     \frac{1}{2} a(v,v)-l(v)   \Big\}
\end{equation}
leads to the following problem formulation.

\begin{prob}[Variational formulation]
Find $u \in V$ such that
\begin{equation}
    a(u,v) = l(v) \quad \forall v \in V.
    \label{weakform}
\end{equation}
\end{prob}

To derive the  corresponding boundary value problem, we recall the following
integration by parts formula, valid in any domain $\sub\subset \Omega$   
\begin{equation}
\begin{aligned}
      \int_\sub& \Mom(w) : \Kappa(v) \,\mathrm{d}x \\
           &= \int_\sub \boldsymbol{\Div} \Mom(w) \cdot \nabla v \,\mathrm{d}x - \int_{\partial \sub} \Mom(w)\boldsymbol{n} \cdot \nabla v\,\mathrm{d}s \\
           &= \int_\sub \A(w)\,v \,\mathrm{d}x + \int_{\partial \sub} \Shear(w) \cdot \boldsymbol{n} \,v\,\mathrm{d}s - \int_{\partial \sub} \Mom(w)\boldsymbol{n} \cdot \nabla v\,\mathrm{d}s ,
\end{aligned}
\label{intparts}
\end{equation}
At the boundary $\partial R$, the correct physical quantities are the components
in the normal $\boldsymbol{n}$ and tangential $\boldsymbol{s}$ directions. 
Therefore, we write 
\begin{equation}
\nabla v=\frac{\partial v}{\partial n} \boldsymbol{n} + \frac{\partial v}{\partial s} \boldsymbol{s}
\end{equation}
and  define the normal shear force and the normal and twisting moments as
\begin{equation}
\begin{aligned}
&\qn(w)=\Shear(w)\cdot \boldsymbol{n},\ \ \ \mnn(w) =  \boldsymbol{n} \cdot  \Mom(w) \boldsymbol{n}, \\  &\mns(w) =\msn(w)=
 \boldsymbol{s} \cdot  \Mom(w) \boldsymbol{n} .
 \end{aligned}
\end{equation}
 With this notation, we can write
 \begin{equation}
 \begin{aligned}
   &  \int_{\partial \sub} \Shear(w) \cdot \boldsymbol{n} \,v\,\mathrm{d}s - \int_{\partial \sub} \Mom(w)\boldsymbol{n} \cdot \nabla v\,\mathrm{d}s
      \\&=  \int_{\partial \sub} \qn(w)v \,\mathrm{d}s - \int_{\partial \sub} \, \Big(\mnn(w)\frac{\partial v}{\partial n}+    \mns(w)\frac{\partial v}{\partial s}\Big)\mathrm{d}s,
\end{aligned}
\end{equation}
and thus rewrite the integration by parts formula \eqref{intparts} as
\begin{equation}  \label{intbypart}
\begin{aligned}
      \int_\sub& \Mom(w) : \Kappa(v) \,\mathrm{d}x \\
                      &= \int_\sub \A(w)\,v \,\mathrm{d}x + 
           \int_{\partial \sub} \qn(w)v \,\mathrm{d}s
           \\&\quad  - \int_{\partial \sub} \, \Big(\mnn(w)\frac{\partial v}{\partial n}+    \mns(w)\frac{\partial v}{\partial s}\Big)\mathrm{d}s.
\end{aligned}
\end{equation}

The key observation for deriving the correct boundary conditions is that, at any boundary point, a
value of $v$ specifies also $\frac{\partial v }{\partial s}$.
Defining the  {\em Kirchhoff shear force} (cf.~\cite{FS,NH,FdV})
\begin{equation}\label{ksf}
\effs(w)= \qn(w) +\frac{ \partial \mns(w)}{\partial s}
\end{equation}
an integration by parts on a smooth part $\smooth$ of $\partial R$ yields
 \begin{equation}\label{IP}
 \int_{  \smooth } \qn(w)v \,\mathrm{d}s
 - \int_{  \smooth} \,  \mns(w)\frac{\partial v}{\partial s} \mathrm{d}s
 = \int_{  \smooth} \effs(w)  \,\mathrm{d}s -{\Big\vert}_a^b\mns(w) v,
\end{equation}
where $a$ and $b$ are the endpoints of $\smooth$.

We are now in the position to state the boundary value problem for the Kirchhoff plate model.
Assuming   a smooth solution $u$ in \eqref{weakform}, we have
  \begin{equation}
      \label{eq:ipb1}
 \begin{aligned}
 a(u,v)=&\int _\Omega    \mathcal{A}(u)  v \, \mathrm{d}x+  \int_{\partial  \Omega} \qn(u)v \,\mathrm{d}s
 \\& -\int_{ \partial \Omega} \, \Big(\mnn(u)\frac{\partial v}{\partial n}+    \mns(u)\frac{\partial v}{\partial s}\Big)\mathrm{d}s.
 \end{aligned}
  \end{equation}
With the combination of clamped, simply supported and free boundary  conditions at
$\partial \Omega = \Gc \cup \Gs \cup \Gf$,  we   have for any $v\in V$,
\begin{equation}
\begin{aligned}
 &\int_{\partial  \Omega} \qn(u)v \,\mathrm{d}s - \int_{ \partial \Omega} \, \Big(\mnn(u)\frac{\partial v}{\partial n}+    \mns(u)\frac{\partial v}{\partial s}\Big)\mathrm{d}s
\\
&= \int_{  \Gf} \qn(u)v \,\mathrm{d}s
 - \int_{  \Gf} \,  \mns(u)\frac{\partial v}{\partial s} \mathrm{d}s
   - \int_{ \Gs\cup \Gf} \, \mnn(u)\frac{\partial v}{\partial n}\mathrm{d}s.
\end{aligned}
\end{equation}
In the final step, we integrate by parts at the free part of the boundary. To this end, let $  \Gf=\cup_{i=1}^{m+1}\Gf^i,$ with $\Gf^i$ smooth.
Integrating by parts over $\Gf^i$ yields
\begin{equation}
    \label{eq:ipb3}
 \int_{  \Gf^i } \qn(u)v \,\mathrm{d}s
 - \int_{  \Gf^i} \,  \mns(u)\frac{\partial v}{\partial s} \mathrm{d}s
 = \int_{  \Gf^i } \effs(u) v \,\mathrm{d}s -{\Big\vert}_{c_{i-1}}^ {c_i}\mns(u) v\,
\end{equation}
where $c_0$ and $ c_{m+1}$ are the end points of $\Gf$ and $c_i,~i=1,\dots, m, $ its successive interior corners.  Combining equations \eqref{eq:ipb1}--\eqref{eq:ipb3}, and noting that $v(c_0)=v(c_{m+1})=0$, gives finally 
\begin{equation}
    \label{eq:ipbfull}
\begin{aligned}
 a(u,v)=\int _\Omega    &\mathcal{A}(u)  v \, \mathrm{d}x   
 - \int_{ \Gs\cup \Gf} \, \mnn(u)\frac{\partial v}{\partial n}\mathrm{d}s
\\&
 + \sum_{i=1}^{m+1} \int_{  \Gf^i }  \effs(u)   v \,\mathrm{d}s - \sum_{i=1}^m\big\{( \mns(u)|_{c_i+} -\mns(u)|_{c_i-}\big\} \,v(c_i),
\end{aligned}
 \end{equation}
 where $M_{ns}(u)|_{c_i \pm} = \lim_{\epsilon \rightarrow 0 + } M_{ns}(u)|_{c_i + \epsilon(c_{i\pm 1} - c_i)}$

Choosing  $v\in V$ in such a way that three of the four terms in \eqref{eq:ipbfull} vanish and the test function in the fourth  term remains arbitrary and repeating this for each term,  we arrive at the following boundary value problem:
\begin{itemize}
\item  In the domain we have the distributional \emph{differential equation}
\begin{equation}
    \mathcal{A}(u) = l \quad \mbox{in } \Omega,
\end{equation}
where $l$ is the distribution defined by \eqref{ldef}.

\item On the {\em clamped} part we have the conditions
\begin{equation}
        u=0\   \ \mbox{and } \ 
 \frac{\partial u}{\partial n}=0 \quad \mbox{on }   \Gc.
\end{equation}

\item On the {\em simply supported} part it holds
\begin{equation}
        u=0\  \  \mbox{and } 
 \mnn(u)=0  \quad \mbox{on }  \Gs.
\end{equation} 

\item On the \emph{free} part  it holds 
\begin{equation}
     \mnn(u)=0\   \mbox{and } 
 \effs(u)=0  \quad \mbox{on }  \Gf^i , \ i=1,\dots, m.
\end{equation}
\item  At the {\em interior corners on the free part}, we have the   matching condition on the twisting moments
\begin{equation}
    \mns(u)|_{c_i+} =\mns(u)|_{c_i-}   \quad \mbox{ for all corners } c_i, \, i=1,\dots,m.
\end{equation}
 \end{itemize}

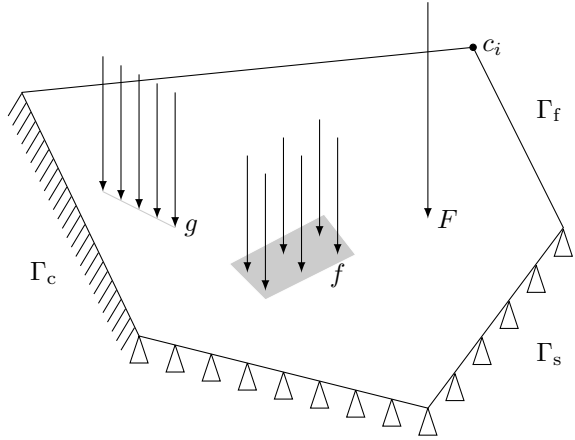
\begin{figure}[h]
    \centering
    \begin{tikzpicture}[scale=1.2]
        \coordinate (A) at (-0.5,0);
        \coordinate (B) at (0.8,-2.7);
        \coordinate (C) at (4,-3.5);
        \coordinate (D) at (5.5,-1.5);
        \coordinate (E) at (4.5,0.5);
        \draw (A) -- (B) -- (C) -- (D) -- (E) -- (A);
        \filldraw[black!20!white,draw=white]  (1.8,-1.9) -- (2.2,-2.3) -- (3.2,-1.8) -- (2.85,-1.35) --
        (1.8,-1.9);
        \draw[-latex] (2,1-1.7) -- (2,-2);
        \draw[-latex] (2.2,0.8-1.7) -- (2.2,-2.2);
        \draw[-latex] (2.4,1.2-1.7) -- (2.4,-1.8);
        \draw[-latex] (2.6,1-1.7) -- (2.6,-2);
        \draw[-latex] (2.8,1.4-1.7) -- (2.8,-1.6);
        \draw[-latex] (3.0,1.2-1.7) -- (3.0,-1.8)node[below] {$f$};
        \draw[-latex] (4.0,1.0) -- (4, -1.4)node[right] {$F$};
        \draw[-latex] (1.2-0.8,0+0.4) -- (1.2-0.8,-1.5+0.4);
        \draw[-latex] (1.2-0.6,0+0.3) -- (1.2-0.6,-1.5+0.3);
        \draw[-latex] (1.2-0.4,0+0.2) -- (1.2-0.4,-1.5+0.2);
        \draw[-latex] (1.2-0.2,0+0.1) -- (1.2-0.2,-1.5+0.1);
        \draw[-latex] (1.2,0) -- (1.2,-1.5) node[right] {$g$};
        \draw[black!20!white] (1.2,-1.5) -- (1.2-0.8,-1.5+0.4);

        \foreach \x in {-0.5,-0.45,...,0.8} {
            \draw (\x,-2.0769*\x-1.05) -- (\x-0.15,-2.0769*\x-1.3);
        }
        \foreach \x in {0.8,1.2,...,4} {
            \draw (\x,-0.25*\x-2.5) -- (\x-0.1,-0.25*\x-2.8) --
            (\x+0.1,-0.25*\x-2.8) -- (\x,-0.25*\x-2.5);
        }
        \foreach \x in {4.3,4.6,...,5.55} {
            \draw (\x,1.33333*\x-8.85) -- (\x-0.1,1.33333*\x-8.85-0.3) --
            (\x+0.1,1.333333*\x-8.85-0.3) -- (\x,1.3333*\x-8.85);
        }
        \draw [fill=black] (E)node[right] {$c_i$} circle (1pt);
        \draw (0,-2) node[anchor=east] {$\Gc$};
        \draw (5.1,-2.9) node[anchor=west] {$\Gs$};
        \draw (5.1,-0.2) node[anchor=west] {$\Gf$};
    \end{tikzpicture}
    \caption{ Definition sketch of a Kirchhoff plate with the different loadings and boundary conditions.}
    \label{fig:plate}
\end{figure}

\section{The finite element method and the a posteriori error analysis} 
\label{sec:fem}

The finite element method is defined on a mesh $\Ch$ consisting of shape regular triangles.
We assume that the point load is applied on a node of the mesh. Further, we
assume that the triangulation is such that the applied line load is on element
edges. We denote the edges in the mesh by  $\Eh$ and divide them into the
following parts: the edges in the interior $\Eh^i$, the edges on the curve of
the line load $\Eh^S \subset \Eh^i$, and the edges on the free and simply supported boundary,
$ \Eh^f $ and $ \Eh^s $, respectively. The conforming finite element space is
denoted by $V_h$. Different choices for $V_h$ are presented in Section~\ref{sec:vh}. Note that we often write $ a \lesssim b$ (or $a \gtrsim b$ ) when $a \leq Cb$ (or  $a \geq Cb$) for some positive constant $C$ independent of the finite element mesh.

\begin{prob}[The finite element method]
    \label{prob:discrete}
    Find $u_h \in V_h$ such that
    \begin{equation}
        a(u_h,v_h) = l(v_h) \quad \forall v_h \in V_h.
    \end{equation}
\end{prob}

Let $K$ and $K'$ be two adjoining triangles with normals $\boldsymbol{n}$ and $\boldsymbol{n}'$, respectively, and  with the common edge  $E=K\cap K'$. On $E$ we define the following jumps 
\begin{equation}\label{momjump}
  \llbracket  \mnn(v)  \rrbracket \vert_E =   \mnn(v)  - M_{n'n'}(v)  
\end{equation}
and
\begin{equation}\label{shearjump}
  \llbracket  \Vn(v)  \rrbracket \vert_E =   \Vn(v)   +V_{n'}(v).
    \end{equation}

In the analysis, we will need the Girault--Scott \cite{GS} interpolation operator   $\Pi_h : V \rightarrow
V_h$ for which  the following estimate  holds
 \begin{equation}
  \label{eq:clement}
  \begin{aligned}
    &\sum_{K \in \Ch} h_K^{-4} \|w-\Pi_h w\|_{0,K}^2 +\sum_{E \in \Eh} h_E^{-1}\|\nabla(w-\Pi_h w)\|_{0,E}^2\\
    &\quad+\sum_{E \in \Eh} h_E^{-3}\|w-\Pi_h w\|_{0,E}^2 \lesssim \|w\|_2^2 \qquad \text{and} \qquad \|\Pi_h w\|_2 \lesssim  \|w\|_2.
\end{aligned}
\end{equation}
Note that the Girault--Scott  interpolant uses point values at the
vertices of the mesh.
We use this property in the proof of Theorem~\ref{thm:aposteriori}  to
derive a proper upper bound for the error in terms of the  edge  residuals.

Next, we formulate an a posteriori estimate for
Problem~\ref{prob:discrete}. 
The local error indicators are the following:
\begin{itemize}
\item The residual on each element
$$ h_K^2 \| \A(u_h) -f \|_{0,K} , \quad K\in \Ch.
$$

\item The  residual of the normal moment jump along interior edges  
$$
h_E^{1/2} \| \llbracket \mnn(u_h) \rrbracket\big\|_{0,E} , \quad E\in  \Eh^i .
$$
\item The residual of the jump in the effective shear force along interior   edges  
$$  h_E^{3/2} \| \llbracket \Vn(u_h) \rrbracket - g \|_{0,E}, \quad   E\in \Eh^S,
$$ 
$$
 h_E^{3/2}\| \llbracket \Vn(u_h)  \rrbracket \|_{0,E},  \quad E\in \Eh^i\setminus \Eh^S.
 $$ 
\item The   normal moment  on edges at the free and simply supported boundaries
$$
h_E^{1/2} \|   \mnn(u_h)  \big\|_{0,E}, \quad E\in   \Eh^f\cup \Eh^s.
$$
\item The  effective shear force along edges at the free boundary
$$
 h_E^{3/2} \|   \Vn(u_h)   \|_{0,E},  \quad E\in   \Eh^f.
 $$ 
\end{itemize}
The global error estimator is then defined through
\begin{equation}
\begin{aligned}
\eta^2=&
    \sum_{K\in \Ch}  h_K^4 \| \A(u_h) -f \|_{0,K}^2 + \sum_{ E\in \Eh^S} h_E^3 \| \llbracket \Vn(u_h)  \rrbracket -g\|_{0,E}^2
    \\
    & \quad 
   + \sum_{E\in \Eh^i\setminus \Eh^S} h_E^3 \| \llbracket \Vn(u_h)  \rrbracket \|_{0,E}^2    
     +\sum_{ E\in  \Eh^i }h_E \| \llbracket \mnn(u_h) \rrbracket \big\|_{0,E}^2
      \\
      & 
      \quad 
   + \sum_{ E\in\Eh^f} h_E^3 \|  \Vn(u_h)   \|_{0,E}^2    
     +\sum_{ E\in  \Eh^f\cup \Eh^s} h_E \|  \mnn(u_h)  \big\|_{0,E}^2.
\end{aligned}
\end{equation} 
\begin{theorem}[A posteriori estimate]
    The following estimate holds
    \begin{equation}
        \|u-u_h\|_2 \lesssim \eta.
\end{equation}
 \label{thm:aposteriori}
\end{theorem}
\begin{proof}
    Let $w = u-u_h$ and $\widetilde{w} := \Pi_h w \in V_h$ be its interpolant.
    In view of the well-known coercivity of the bilinear form $a$ and Galerkin orthogonality, we have
    \begin{equation}
    \begin{aligned}
                     \|u-u_h\|_2^2 
                    &\lesssim a(u-u_h,w) 
                    = a(u-u_h,w-\widetilde{w}) 
                    \\
                    &= l(w-\widetilde{w})-a(u_h,w-\widetilde{w}) .
    \end{aligned}
    \end{equation}
   Since  $x_0$ is a mesh node and  the interpolant uses nodal values, we have
    \begin{equation}
F \big(w(x_0)-\widetilde{w}(x_0)\big)=0,
    \end{equation} 
    and hence
    \begin{equation}
    l(w-\widetilde{w})=(f, w-\widetilde{w})      +\langle g, w-\widetilde{w}  
                    \rangle_S .
    \end{equation}
From integration by parts over the element edges, using the fact that the
interpolant uses values at the nodes, it then follows that
   \begin{equation}
      \begin{aligned}
                    &\|u-u_h\|_2^2 \\
                    &\lesssim  (f, w-\widetilde{w})      +\langle g, w-\widetilde{w}  \rangle_S-a(u_h,w-\widetilde{w}) 
                    \\
                    &= (f, w-\widetilde{w})      +\langle g, w-\widetilde{w}  
                    \rangle_S 
                    \\& \qquad 
                    -  \sum_{K \in \Ch} \big\{\left(\A(u_h),w-\widetilde{w}\right)_K +  
                    \langle\qn(u_h) ,w-\widetilde{w}\rangle_{\partial K}  \\
                   &\qquad   \qquad -
                 \langle \mns(u_h),\tfrac{\partial}{\partial s} (w-\widetilde{w})\rangle_{\partial K} 
                 -   \langle \mnn(u_h),\tfrac{\partial}{\partial n} (w-\widetilde{w})\rangle_{\partial K}\big\}
                    \\
                      &= (f, w-\widetilde{w})      +\langle g, w-\widetilde{w}  \rangle_S
                      \\&
                      \qquad
                      -  \sum_{K \in \Ch} \big\{\left(\A(u_h),w-\widetilde{w}\right)_K +  
                    \langle\Vn(u_h) ,w-\widetilde{w}\rangle_{\partial K}
                    \\& \qquad - \langle \mnn(u_h),\tfrac{\partial}{\partial n} (w-\widetilde{w})\rangle_{\partial K} \big\}
                    .
    \end{aligned}
    \end{equation}
    Regrouping and recalling definitions \eqref{momjump} and \eqref{shearjump},  yields
    \begin{equation}
     \begin{aligned}
                    &\|u-u_h\|_2^2
                     \\
                      \lesssim     &\sum_{K \in \Ch} \left(f-\A(u_h),w-\widetilde{w}\right)_K  
                      \\ & \quad
                   -\sum_{E\in \Eh^S}  \langle \llbracket \Vn(u_h) \rrbracket - g ,w-\widetilde{w}\rangle_{E} -
                   \sum_{E\in \Eh^i \setminus \Eh^S} \langle  \llbracket \Vn(u_h)\rrbracket  ,w-\widetilde{w}\rangle_{E}
                    \\
                    &\quad - \sum_{E\in\Eh^i} \langle \llbracket \mnn(u_h)\rrbracket, \tfrac{\partial}{\partial n_E} (w-\widetilde{w})\rangle_{E}                     \\
                    &\quad -    \sum_{E\in\Eh^f} \langle  \Vn(u_h) ,w-\widetilde{w}\rangle_{E}
 - \sum_{ E\in  \Eh^f\cup \Eh^s} \langle \mnn(u_h), \tfrac{\partial}{\partial n_E} (w-\widetilde{w})\rangle_{E} 
                    .
    \end{aligned}
    \end{equation}
    The asserted a posteriori estimate now follows by applying the Cauchy--Schwarz inequality and the interpolation estimate \eqref{eq:clement}.
  \end{proof}
 
Instead of the jump terms in the estimator  $\eta$, we could consider the normal and twisting moment jumps
   $$
h_E^{1/2} \|  \llbracket  \mnn(u_h)\rrbracket  \big\|_{0,E}, \quad h_E^{1/2} \|  \llbracket  \mns(u_h) \rrbracket   \big\|_{0,E}, $$
   and the normal shear force jumps
   $$  h_E^{3/2}\| \llbracket \qn(u_h)  \rrbracket \|_{0,E},  \quad 
 h_E^{3/2} \| \llbracket \qn(u_h) \rrbracket - g \|_{0,E}.
 $$ 
In this case we cannot, however, prove the efficiency, i.e. the lower bounds.

Next, we will consider the question of efficiency.
Let $f_h\in V_h$ be the interpolant of $f$ and define
\begin{equation}
\osc_K(f) = h_K^2\Vert f-f_h\Vert_{0,K}.
\end{equation}
Similarly, for   a polynomial approximation  $g_h$ of $g$ on $E \subset S$ we define
\begin{equation}
\osc_E(g) = h_E^{3/2}\Vert g-g_h\Vert_{0,E}.
\end{equation}
 In the following theorem,  $\omega_E$ stands for the union of elements sharing an edge $E$.
In its proof, we will adopt some of the techniques used in \cite{Gudi-Porwal}.
\begin{theorem}[Lower bounds]
    For all $v_h \in V_h$ it holds
    \begin{align}
        h_K^2 \| \A(v_h) - f\|_{0,K} &\lesssim \|u-v_h\|_{2,K} + \osc_K(f), \ K\in \Ch, \label{eq:lowboundlocalint}
         \\
       \label{eq:edgebound} 
       h_E^{1/2}\| \llbracket \mnn(v_h) \rrbracket \|_{0,E} &\lesssim \|u-v_h\|_{2,\omega_E}  
                                                                                                          + \sum_{K \subset \omega_E} \osc_K(f), \quad E\in \Eh^i,
  \\
       \label{eq:edgebound2} 
       h_E^{3/2}\| \llbracket  \Vn(v_h) \rrbracket \|_{0,E} 
       & \lesssim \|u-v_h\|_{2,\omega_E} 
 + \sum_{K \subset \omega_E} \osc_K(f), \quad E\in \Eh^i\setminus \Eh^S,
 \\
    \label{eq:edgebound3}
 h_E^{3/2}\| \llbracket  \Vn(v_h)\rrbracket  -g\|_{0,E} 
       & \lesssim \|u-v_h\|_{2,\omega_E} 
 + \sum_{K \subset \omega_E} \osc_K(f)+\osc_E(g), \ E\in  \Eh^S,
       \\
 \label{eq:edgebound4}
  h_E^{1/2}\| \mnn(v_h)   \|_{0,E} &\lesssim \|u-v_h\|_{2,\omega_E}  
                                                                                                          + \sum_{K \subset \omega_E} \osc_K(f), \quad E\in \Eh^f\cup \Eh^s,
        \\                                                                                                  
       \label{eq:edgebound5}
        h_E^{3/2}\|   \Vn(v_h)   \|_{0,E} 
       & \lesssim \|u-v_h\|_{2,\omega_E} 
    + \sum_{K \subset \omega_E} \osc_K(f), \quad E\in \Eh^f.
\end{align}
    \end{theorem}
\begin{proof}
    Denote by $b_K \in P_6(K)$ the sixth order bubble that, together with its
    first-order derivatives, vanishes on $\partial K$, i.e.~let $b_K =
    (\lambda_{1,K} \lambda_{2,K} \lambda_{3,K})^2$, where $\lambda_{j,K}$ are
    the barycentric coordinates for $K$. Then we define
    \begin{equation}
        \gamma_K = b_K h_K^4(\A(v_h)-f_h) ~ \text{in $K$} \quad \text{and} \quad \gamma_K=0 ~\text{in $\Omega \setminus K$},
    \end{equation}
    for  $v_h \in V_h$. The problem statement gives
    \begin{equation}
        a_K(u,\gamma_K) = (f,\gamma_K)_K,
    \end{equation}
    where $ a_K(u,\gamma_K) = \int_K \Mom(u) : \Kappa(\gamma_K) \,\mathrm{d}x $. We have
    \begin{equation}
        \label{eq:lemp1}
        \begin{aligned}
            h_K^4 \|\A(v_h)-f_h\|_{0,K}^2
            &\lesssim h_K^4 \| \sqrt{b_K}(\A(v_h)-f_h) \|_{0,K}^2 \\
            &= (\A(v_h)-f_h,\gamma_K)_K \\
            &= (\A(v_h),\gamma_K)_K - (f,\gamma_K)_K + (f-f_h,\gamma_K)_K \\ 
            &= a_K(v_h-u,\gamma_K)+(f-f_h,\gamma_K)_K. 
        \end{aligned}
    \end{equation}
The local bound \eqref{eq:lowboundlocalint} now follows from applying the continuity of $a$, the Cauchy--Schwarz inequality and inverse estimates.

    Next, consider inequality \eqref{eq:edgebound}. Suppose  $E=K_1 \cap K_2$ for  the triangles $K_1$ and $K_2$; thus  $\omega_E=K_1\cup K_2$. Let $\lambda_E \in P_1(\omega_E)$ be the linear
    polynomial satisfying
    \begin{equation}
      \lambda_E\vert_E=0,\ \text{ and }   \  \frac{\partial \lambda_E}{\partial n_E} =1,
    \end{equation}
  and let $p_1$ be the polynomial that satisfies $p_1 \vert_ E= \llbracket
  M_{nn}(v_h) \rrbracket\vert_E $ and $\frac{\partial p_1}{\partial n_E} \vert_E=
  0$. Moreover, let $p_2 \in P_8(\omega_E)$ be the eight-order bubble that
  takes value one at the midpoint of the edge $E$ and, together with its first-order
  derivatives, vanishes on $\partial \omega_E$.
    Define  $w = \lambda_E p_1 p_2$. Since    
    \begin{equation}
        \frac{\partial w}{\partial n_E}\Big|_E = \frac{\partial \lambda_E}{\partial n_E} \llbracket M_{nn}(v_h) \rrbracket p_2=\llbracket M_{nn}(v_h) \rrbracket p_2, 
        \label{wnrelation}
    \end{equation}
      scaling yields the equivalence
    \begin{equation}
    \begin{aligned}
           \|\llbracket M_{nn}(v_h) \rrbracket\|_{0,E}^2 &\approx   \Big\|  \frac{\partial w}{\partial n_E}\Big\|_{0,E}^2
            \approx \|\sqrt{ p_2}\,\llbracket M_{nn}(v_h) \rrbracket\|_{0,E}^2
           \\&
   = \big\langle \llbracket M_{nn}(v_h)\rrbracket,\tfrac{\partial w}{\partial n_E}\big\rangle_E.
   \label{equiv1}
   \end{aligned}
    \end{equation} 
    Furthermore, since 
     \begin{equation}\label{rs1}
     \frac{\partial w}{\partial s}\Big|_{E} = 0,\quad  w\vert_{E \cup\partial \omega_E}=0 \quad \text{and} \quad  \nabla w\vert_{\partial \omega_E} =\mathbf{0},
    \end{equation}
  the integration by parts formula  \eqref{intbypart} yields
     \begin{equation}
         \big\langle \llbracket M_{nn}(v_h)\rrbracket,\tfrac{\partial w}{\partial n_E}\big\rangle_E 
        \\                                                  =-\int_{\omega_E}\Mom(v_h):\Kappa(w)\,\mathrm{d}x + (\A(v_h),w)_{\omega_E}.   
        \label{newintbyparts} \end{equation}
        Extending $w$ by zero to $\Omega\setminus \omega_E$,  we obtain from the problem statement \eqref{weakform}  
        \begin{equation}
        \int_{\omega_E}\Mom(u):\Kappa(w)\,\mathrm{d}x -(f,w)_{\omega_E}=0.
        \end{equation}
        Hence, using 
         the Cauchy--Schwarz inequality, we get from \eqref{newintbyparts}
        \begin{equation}\label{B} 
    \begin{aligned}
          \big\langle \llbracket M_{nn}&(v_h)\rrbracket, \tfrac{\partial w}{\partial n_E}\big\rangle_E 
        \\
                                                                              &=  \int_{\omega_E}\Mom(u-v_h):\Kappa(w)\,\mathrm{d}x+  (\A(v_h)-f,w)_{\omega_E} \\
                                                          &\lesssim \|u-v_h\|_{2,\omega_E} |w|_{2,\omega_E} + \|\A(v_h)-f\|_{0,\omega_E} \|w\|_{0,\omega_E}.
                                                            \end{aligned}
    \end{equation}
    By scaling, one easily shows that
            \begin{equation}\label {C}
        |w|_{2,\omega_E}  \lesssim    h_E^{-1/2}\Big\|  \frac{\partial w}{\partial n_E}\Big\|_{0,E}\quad \text{and }  \ \|w\|_{0,\omega_E} \lesssim  h_E^{3/2} \Big\|  \frac{\partial w}{\partial n_E}\Big\|_{0,E}.
        \end{equation}
    The estimate \eqref{eq:edgebound} 
     then follows from \eqref{wnrelation},  \eqref{equiv1}, \eqref{B}, \eqref{C}, and the already proved bound \eqref{eq:lowboundlocalint}.

    Since~\eqref{eq:edgebound2} follows from~\eqref{eq:edgebound3} with
    $g=0$, we prove the latter.  Due to the regularity condition imposed on the
    mesh there exists for each edge $E$ a symmetric pair of smaller triangles
    $(K^\prime_1,K^\prime_2)$ that satisfy $  \omega^\prime_E=K_1^\prime\cup
    K_2^\prime \subset \omega_E$, see Fig.~\ref{fig:rhomb}.      Let
    $w^\prime = p_2^\prime (\llbracket  \Vn(v_h)\rrbracket -g_h) $ where
    $p_2^\prime$ is the eight-order bubble that takes value one at the midpoint
    of $E$ and, together its first-order derivatives, vanishes on $\partial
    \omega_E^\prime$.
    By the norm equivalence, we first have
   \begin{equation}  \label{one}
        \|\llbracket  \Vn(v_h)\rrbracket   -g_h\|_{0,E}^2   \approx  \|   w^\prime \|_{0,E}^2
        \lesssim \langle \llbracket  \Vn(v_h)\rrbracket  -g_h, w^\prime      \rangle_E .
    \end{equation}
    Next, we write 
       \begin{equation}  \label{two}
         \langle \llbracket  \Vn(v_h)\rrbracket  -g_h, w^\prime      \rangle_E
        =
         \langle \llbracket  \Vn(v_h)\rrbracket  -g, w^\prime      \rangle_E +  \langle g -g_h, w^\prime      \rangle_E.
    \end{equation}
     Due to  symmetry, $\frac{\partial w^\prime}{\partial n}|_E = 0$, and hence \eqref{intbypart} and \eqref{IP} give
    \begin{equation} \begin{aligned}
        &  \langle \llbracket  \Vn(v_h)\rrbracket  -g, w^\prime      \rangle_E = \langle \llbracket  \Vn(v_h)\rrbracket, w^\prime      \rangle_E-\langle   g, w^\prime      \rangle_E
                \\
                &=\int_{\omega_E^\prime}\Mom(v_h):\Kappa(w^\prime )\,\mathrm{d}x-  (\A(v_h),w^\prime )_{\omega_E^\prime} 
                -\langle   g, w^\prime      \rangle_E.
                \end{aligned}
                \end{equation}
               Extending $w^\prime$ by zero to $\Omega\setminus \omega^\prime_E$, the variational form \eqref{weakform} implies that
       \begin{equation}
        \int_{\omega_E^\prime}\Mom(u):\Kappa(w^\prime)\,\mathrm{d}x -(f,w^\prime)_{\omega_E^\prime}-\langle g,w^\prime\rangle_E=0  \, .
        \end{equation}
        Hence, 
         \begin{equation}
       \langle \llbracket  \Vn(v_h)\rrbracket  -g, w^\prime      \rangle_E=  \int_{\omega_E^\prime}\Mom(v_h-u):\Kappa(w^\prime)\,\mathrm{d}x +(f-\A(v_h) ,w^\prime)_{\omega_E^\prime}
        \end{equation}
        and the Cauchy--Schwarz inequality, scaling estimates and \eqref{eq:lowboundlocalint} give
        \begin{equation}\label{three}
        \begin{aligned}
     & \langle \llbracket  \Vn(v_h)\rrbracket  -g, w^\prime      \rangle_E    \\
        &\lesssim h_E^{-3/2} \left( \|u-v_h\|_{2,\omega_E^\prime} +h_K^2 \|\A(v_h)-f\|_{0,\omega_E^\prime} \right) \| w^\prime \|_{0,E} \\
        &\lesssim h_E^{-3/2} \Big( \|u-v_h\|_{2,\omega_E} + \sum_{K \in \omega_E} \osc_K(f) \Big)
         \| w^\prime \|_{0,E}  .
    \end{aligned}    
    \end{equation}
 The asserted estimate then follows from \eqref{one}, \eqref{two} and \eqref{three}.

     The estimates  \eqref{eq:edgebound4},
 \eqref{eq:edgebound5}, are proved similarly to the bounds  \eqref{eq:edgebound} and  \eqref{eq:edgebound2}, respectively.
\end{proof}
The above estimates provide  the following global bound.
\begin{theorem}\label{globlow}  
It holds
    \begin{equation}
        \eta \lesssim \|u-u_h\|_2 + \osc(f)+ \osc(g),
\end{equation}
where 
\begin{equation}
    \osc(f) = \sqrt{ \sum_{K \in \Ch} \osc_K(f)^2 }\quad  \mbox{ and }  \quad  \osc(g) = \sqrt{ \sum_{E\in  \Eh^S} \osc_E(g)^2}. 
\end{equation}
\end{theorem}

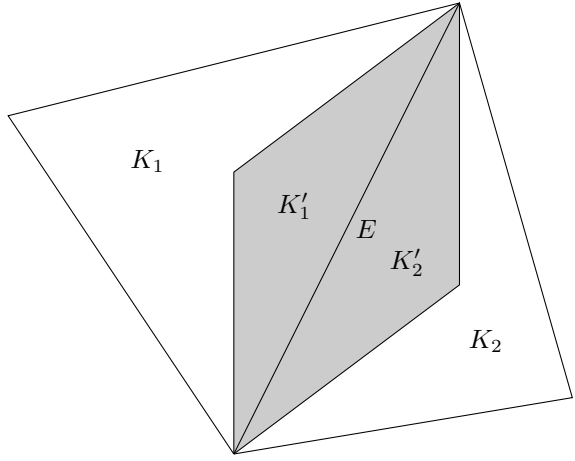
\begin{figure}[h]
    \centering
    \begin{tikzpicture}[scale=1.5]
        \draw (-1,-2) -- (-3,1) -- (1,2) -- (2,-1.5) -- (-1,-2);
        \filldraw[black!20!white,draw=black] (-1,-2) -- (-1,0.5) -- (1,2) -- (1,-0.5) -- (-1,-2);
        \draw (-1,-2) -- (1,2);
        \draw (0,0) node[anchor=west] {$E$};
        \draw (1,-1) node[anchor=west] {$K_2$};
        \draw (0.3,-0.3) node[anchor=west] {$K_2^\prime$};
        \draw (-2,0.6) node[anchor=west] {$K_1$};
        \draw (-0.7,0.2) node[anchor=west] {$K_1^\prime$};
    \end{tikzpicture}
    \caption{A depiction of the sets $\omega_E$ (the entire polygon) and $\omega^\prime_E$ (the grey area). The triangles $K_1^\prime$ and $K_2^\prime$ are symmetric with respect to the edge $E$.}
    \label{fig:rhomb}
\end{figure}

\section{The choice of $V_h$}
\label{sec:vh}

Let us briefly discuss some possible choices of conforming finite elements for
the plate bending problem. Each choice consists of a polynomial space
$\mathcal{P}$ and of a set of $N$ degrees of freedom defined through a functional $\mathcal{L} : C^\infty
\rightarrow \mathbb{R}$. We denote by $\boldsymbol{x}^k$,~$k\in \{1,2,3\}$, the vertices of the
triangle and by $\boldsymbol{e}^k$,~$k\in \{1,2,3\}$, the midpoints of the edges, i.e.
\begin{equation}
    \boldsymbol{e}^1 = \frac12(\boldsymbol{x}^1+\boldsymbol{x}^2), \quad \boldsymbol{e}^2 = \frac12(\boldsymbol{x}^2+\boldsymbol{x}^3), \quad \boldsymbol{e}^3 =
    \frac12(\boldsymbol{x}^1+\boldsymbol{x}^3).
\end{equation}

The simplest $H^2$-conforming triangular finite element that is locally $H^4(K)$
in each $K$ is the Bell triangle.

\begin{defn}[Bell triangle, $N=18$]
    \begin{align}
        \mathcal{P} &= \{ p \in P_5(K) : \tfrac{\partial p}{\partial n} \in P_3(E)~\forall E \subset K \} \\
        \mathcal{L}(w) &= \begin{cases}
        w(\boldsymbol{x}^k), & \text{for $1\leq k \leq 3$,} \\[0.1cm]
        \frac{\partial w}{\partial x_i}(\boldsymbol{x}^k), & \text{for $1 \leq k \leq 3$ and $1 \leq i \leq 2$,} \\[0.1cm]
        \frac{\partial^2 w}{\partial x_i \partial x_j}(\boldsymbol{x}^k), & \text{for $1 \leq k \leq 3$ and $1 \leq i, j \leq 2$.}
        \end{cases}
    \end{align}
\end{defn}

Even though the polynomial space associated with the Bell triangle is not the whole $P_5(K)$ it is still larger
than $P_4(K)$. This can in some cases complicate the implementation.
Moreover, the asymptotic interpolation estimates for $P_5(K)$ are not obtained.
This can be compensated by adding three degrees of freedom at the midpoints of the
edges of the triangle and increasing accordingly the size of the polynomial space.

\begin{defn}[Argyris triangle, $N=21$]
    \begin{align}
        \mathcal{P} &= P_5(K), \\
        \mathcal{L}(w) &= \begin{cases}
        w(\boldsymbol{x}^k), & \text{for $1\leq k \leq 3$,} \\[0.1cm]
        \frac{\partial w}{\partial x_i}(\boldsymbol{x}^k), & \text{for $1 \leq k \leq 3$
    and $1 \leq i \leq 2$,} \\[0.1cm]
    \frac{\partial^2 w}{\partial x_i \partial x_j}(\boldsymbol{x}^k), & \text{for $1 \leq k \leq 3$
    and $1 \leq i, j \leq 2$,} \\[0.1cm]
    \frac{\partial w}{\partial n}(\boldsymbol{e}^k), & \text{for $1 \leq
                k \leq 3$.}
        \end{cases}
    \end{align}
\end{defn}

The Argyris triangle can be further generalized to higher-order polynomial
spaces, cf.~P.~{\v{S}}ol{\'{i}}n~\cite{Solin}. Triangular macroelements such as
the Hsieh--Clough--Tocher triangle are not locally $H^4(K)$ and therefore
additional jump terms are present inside the elements.  Various conforming
quadrilateral elements have been proposed in the literature for the plate
bending problem cf.~Ciarlet~\cite{Ciarlet}. The proofs of the  lower bound  that we presented
do not directly apply to quadrilateral elements,  but the  techniques can be adapted to them as well.

\section{Numerical results}

In our examples, we will use the fifth degree Argyris triangle. On a uniform mesh for a solution $u\in H^r (\Omega)$, with $r\geq 2$,  we thus have the error estimate \cite{Ciarlet} 
\begin{equation}
\Vert u-u_h\Vert_2 \lesssim h^s \vert u \vert_r,
\end{equation}
with $s=\min\{r-2,4\}$.
Since the mesh length is related to the number of degrees of freedom $N$ by  $h\sim N^{-1/2}$ on a uniform mesh,  we can also write
\begin{equation}
\Vert u-u_h\Vert_2 \lesssim N^{-s/2} \vert u \vert_r.
\end{equation}
If the solution is smooth, say $r\geq 6$, we thus have the estimates 
\begin{equation}
\Vert u-u_h\Vert_2 \lesssim h^4\  \mbox{ and }  \ \Vert u-u_h\Vert_2 \lesssim N^{-2}.
\end{equation}
In fact, the  rate $ N^{-2}$ is  optimal also on a general mesh since, except for a polynomial solution, it holds \cite{BabSca,BabAz}
\begin{equation}
  \Vert u-u_h\Vert_2 \gtrsim N^{-2}.
\end{equation}

In the adaptive  computations we use the following strategy for marking the
elements that will be refined~\cite{VerfurthI}.

\begin{alg}
    Given a partition $\Ch$, error indicators $\eta_K$, $K \in  \Ch$ and a
    threshold $\theta \in (0,1)$, mark $K$ for refinement if $\eta_K \geq
    \theta \max_{K^\prime \in \Ch} \eta_{K^\prime}$.
\end{alg}

The parameter $\theta$ has an effect on the portion of elements that are
marked, i.e.~for $\theta = 0$ all elements are marked and for $\theta = 1$ only
the element with the largest error indicator value is marked. We simply take
$\theta = 0.5$ which has proven to be a feasible choice in most cases.

The set of marked elements are refined using 
Triangle~\cite{shewchuk1996triangle}, version 1.6, by requiring additional
vertices at the edge midpoints of the marked elements and by allowing the mesh
generator to improve mesh quality through extra vertices. The default minimum
interior angle constraint of 20 degrees is used.

The regularity of the solution depends on the regularity of the load and the corner singularities, cf. \cite{BR}. Below we consider two sets of problems, one where the regularity is mainly restricted by the load, and another one where the load is uniform and the corner singularities dominate.

\subsection{Square plate, Navier solution}

A classical series solution to the Kirchhoff plate bending problem, the Navier
solution \cite{TWK}, in the special case of a unit square with simply supported boundaries
and the loading
\begin{equation}
    f(\boldsymbol{x}) = \begin{cases}
        f_0, & \text{if $\boldsymbol{x} \in [\tfrac12-c, \tfrac12+c] \times [\tfrac12-d, \tfrac12+d]$},\\
        0, & \text{otherwise,}
    \end{cases}
\end{equation}
reads
\begin{equation}
    u(x,y) = \frac{16 f_0}{D \pi^6} \sum_{m=1}^\infty \sum_{n=1}^\infty \frac{\sin{\frac{m\pi}{2}}\sin{\frac{n\pi}{2}}\sin{m \pi c}\sin{n\pi d}}{mn(m^2+n^2)^2} \sin{m \pi x} \sin{n \pi y}.
\end{equation}
In the limit $c \longrightarrow 0$ and $2 c f_0 \longrightarrow g_0$ we get the
line load solution
\begin{equation}
    u(x,y) = \frac{8 g_0}{D \pi^5} \sum_{m=1}^\infty \sum_{n=1}^\infty \frac{\sin{\frac{m\pi}{2}}\sin{\frac{n\pi}{2}} \sin{n\pi d}}{n(m^2+n^2)^2} \sin{m \pi x} \sin{n \pi y},
\end{equation}
and in the limit $c, d \longrightarrow 0$ and $4 c d f_0 \longrightarrow F_0$ we obtain
the point load solution
\begin{equation}
    u(x,y) = \frac{4 F_0}{D \pi^4} \sum_{m=1}^\infty \sum_{n=1}^\infty \frac{\sin{\frac{m\pi}{2}}\sin{\frac{n\pi}{2}}}{(m^2+n^2)^2} \sin{m \pi x} \sin{n \pi y}.
\end{equation}
From the series we can infer that the solution is in  $H^{3-\epsilon}(\Omega)$, $H^{7/2-\epsilon}(\Omega)$ and $H^{9/2-\epsilon}(\Omega)$, for any $\epsilon>0$, for the point load, line load and the square load, respectively.  
in the three cases. On a uniform mesh, one should thus observe the convergence rates $  N^{-0.5}, \,  N^{-0.75} $, and $ N^{-1.25}$.

An unfortunate property of the series solutions is that the partial
sums converge very slowly. This makes computing the difference between
the finite element solution and the series solution in $H^2(\Omega)$ and $L^2(\Omega)$-norms
a challenging task since the finite element solution quickly ends up
being more accurate than any reasonable partial sum. In fact, the "exact" series solution is practically useless, for example,
 for computing the shear force which is an important design parameter.

The $H^2(\Omega)$-norm is equivalent to the energy norm,
\begin{equation}
\vertiii{v}= \sqrt{a(v,v)},
\end{equation}
with which the error  is straightforward to compute.
In view of the  Galerkin
orthogonality and symmetry, one obtains
    \begin{equation}
        \vertiii{u-u_h}^2 = a(u-u_h,u) 
                          = l(u-u_h) ,                           
    \end{equation}
    i.e. the error is given by
     \begin{equation}
        \vertiii{u-u_h}= \sqrt{
                            l(u-u_h)}  .                         
    \end{equation} 
 This is especially useful for the point load for which 
  \begin{equation}
        \vertiii{u-u_h}= \sqrt{
                          F_0\big(u(\tfrac12,\tfrac12)-u_h(\tfrac12,\tfrac12)\big)}  .                         
    \end{equation} 
Evaluating the series solution at the point of maximum deflection gives \cite{TWK} 
\begin{equation}
    \label{eq:maxseries}
    \begin{aligned}
        u(\tfrac12,\tfrac12) &= \frac{4 F_0}{D \pi^4} \sum_{m=1}^\infty \sum_{n=1}^\infty \frac{(\sin{\frac{m\pi}{2}}\sin{\frac{n\pi}{2})^2}}{(m^2+n^2)^2} \\
                             &= \frac{4 F_0}{D \pi^4} \sum_{m=1}^\infty \left(\sin{\frac{m\pi}{2}}\right)^2 \sum_{n=1}^\infty \frac{\left(\sin{\frac{n\pi}{2}}\right)^2}{(m^2+n^2)^2} \\
                             &= \frac{F_0}{2 D \pi^3} \sum_{m=1}^\infty \frac{(\sin{\frac{m\pi}{2}})^2(\sinh{m\pi}-m\pi)}{m^3(1+\cosh{m\pi})}.
    \end{aligned}
\end{equation}

We first consider a point load with $F_0=1$, $d=1$, $E=1$ and $\nu=0.3$,
and compare the true error with the estimator $\eta$. In this case, we have the
approximate maximum displacement $u(\tfrac12,\tfrac12) \approx 0.1266812$,  computed by evaluating and summing the first 10 million terms of the series
\eqref{eq:maxseries}.  Starting with an initial mesh shown in
Fig.~\ref{fig:mesh-pointload}, we repeatedly mark and refine the mesh to obtain
a sequence of meshes, see Fig.~\ref{fig:estim-pointload} where the values of the
elementwise error estimators are depicted for four consecutive meshes. Note that the estimator and the adapted
marking strategy initially refine heavily in the neighborhood of the point load
as one might expect based on the regularity of the solution in the vicinity of
the point load.

In addition to the adaptive strategy, we solve the problem using a uniform mesh
family where we repeatedly split each triangle into four subtriangles starting
from the initial mesh of Fig.~\ref{fig:mesh-pointload}. The
energy norm error and $\eta$ versus the number of degrees of freedom $N$ are
plotted in Fig.~\ref{fig:pointloadgraph}.  The results show that the adaptive
meshing strategy improves  significantly  the rate of convergence in the energy
norm. In Fig.~\ref{fig:pointloadgraph},  we have also plotted, for reference, the slopes corresponding to the expected convergence
rate $O(N^{-0.5})$ for uniform refinement and the optimal
convergence rate for $P_5$ elements, $O(N^{-2})$.

In Fig.~\ref{fig:pointloadgraph} is is further revealed that the energy norm
error and the estimator $\eta$ follow similar trends. This is exactly what one
would expect given that the estimator is an upper and lower bound for the true
error modulo an unknown constant. This is better seen by drawing the normalized
ratio $\eta$ over $\vertiii{u-u_h}$, see Fig.~\ref{fig:pointloadefficiency}. 
Since the estimator correctly follows the true error and an accurate
computation of norms like $\|u-u_h\|_2$ is expensive, the rest of the
experiments document only the values of $\eta$ and $N$ for the purpose of
giving idea of  the convergence rates.

We continue with the line load case taking $g_0=1$ and $d=\frac13$, and using the same
material parameter values as before.  The initial and final meshes are shown in
Fig.~\ref{fig:mesh-lineload}. The estimator can be seen to primarly focus on
the end points of the line load. The values of $\eta$ and $N$ are visualized in Fig.~\ref{fig:lineloadgraph}
together with the expected and the optimal rates of convergence. Again the
adaptive strategy improves the convergence of the total error in comparison to the uniform
refinement strategy. The local error estimators and the adaptive process are
presented in Fig.~\ref{fig:estim-lineload}.

We finish this subsection by solving the square load case with $f_0=1$,
$c=d=\frac13$ and the same material parameters as before. The initial and the final
meshes are shown in Fig.~\ref{fig:mesh-squareload}. The convergence
rates are visualized in Fig.~\ref{fig:squareloadgraph} and the local error
estimators in Fig.~\ref{fig:estim-squareload}. An improvement in the
convergence rate is again visible in the results.

\begin{figure}
    \includegraphics[width=0.49\textwidth]{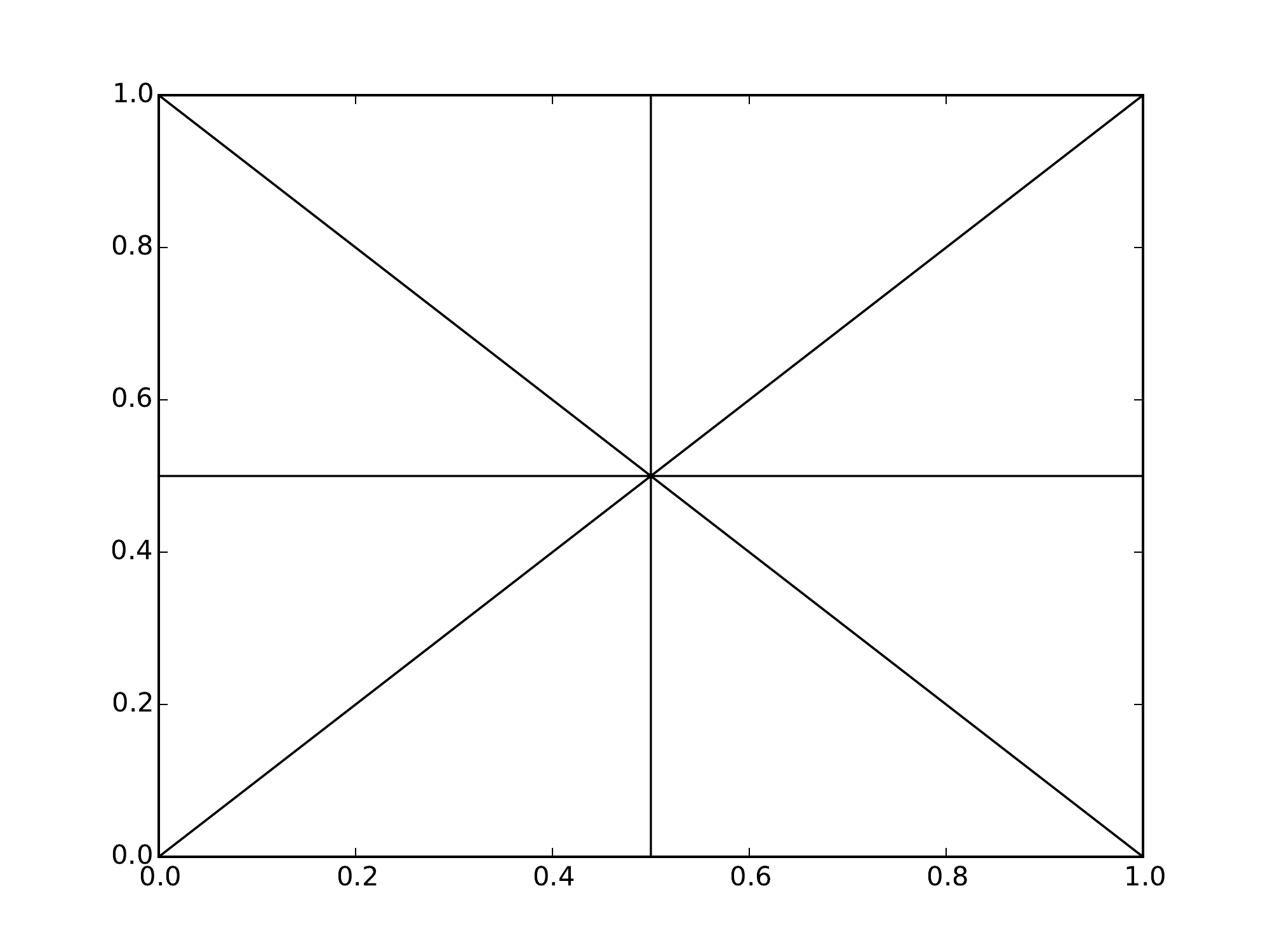}
    \includegraphics[width=0.49\textwidth]{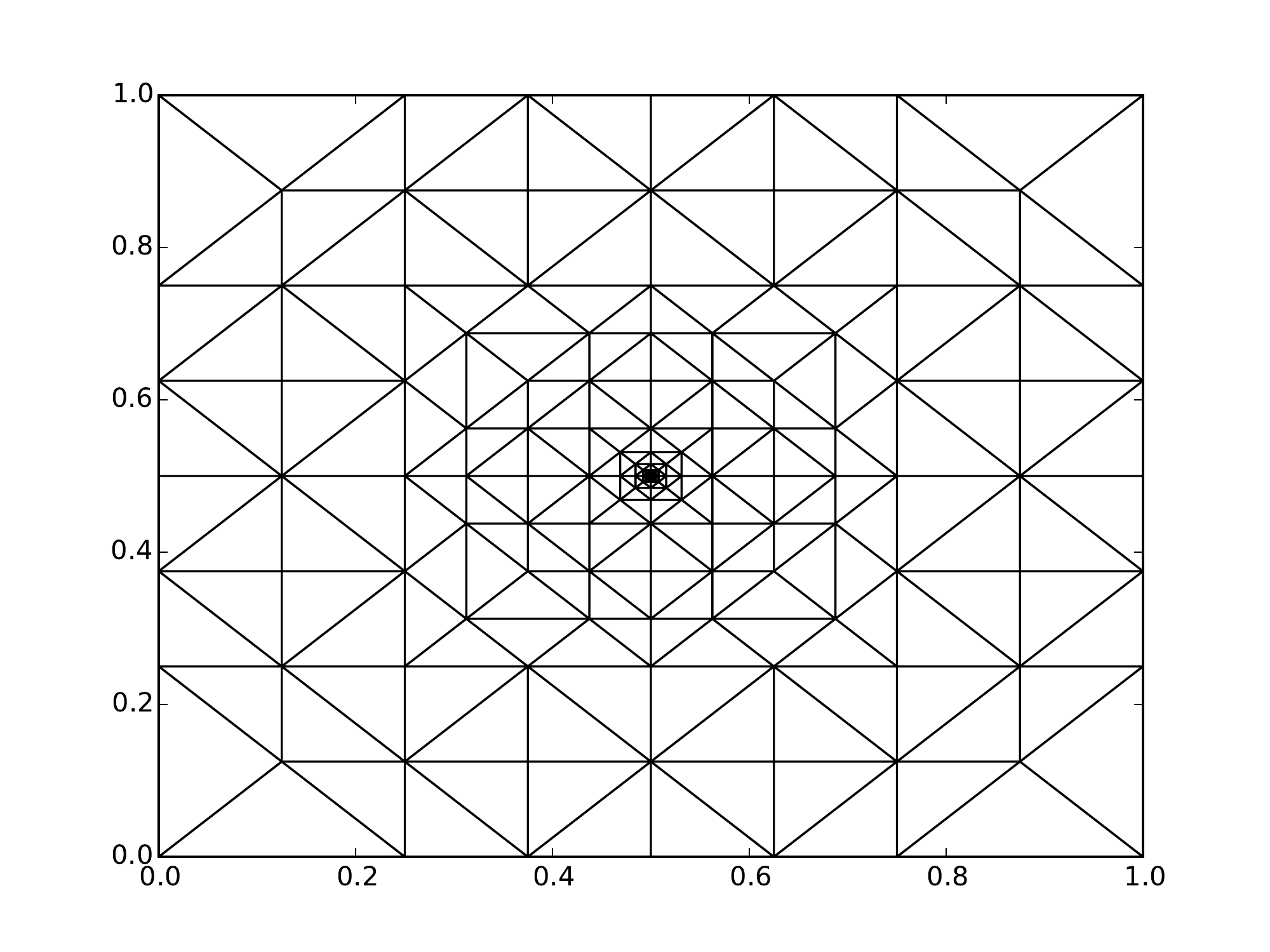}
    \caption{Initial and six times refined meshes in the point load case.}
    \label{fig:mesh-pointload}
\end{figure}

\begin{figure}
    \includegraphics[width=0.49\textwidth]{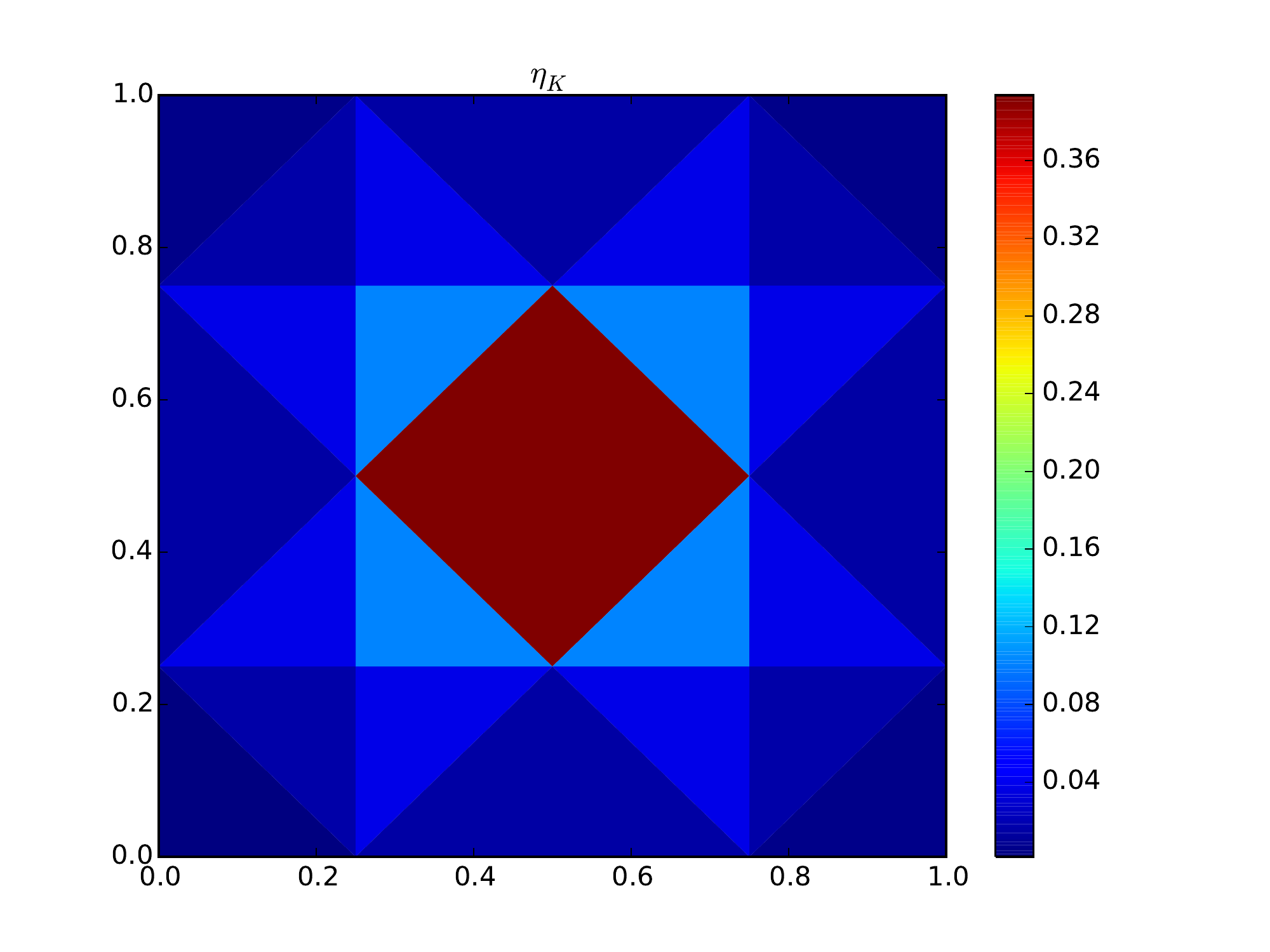}
    \includegraphics[width=0.49\textwidth]{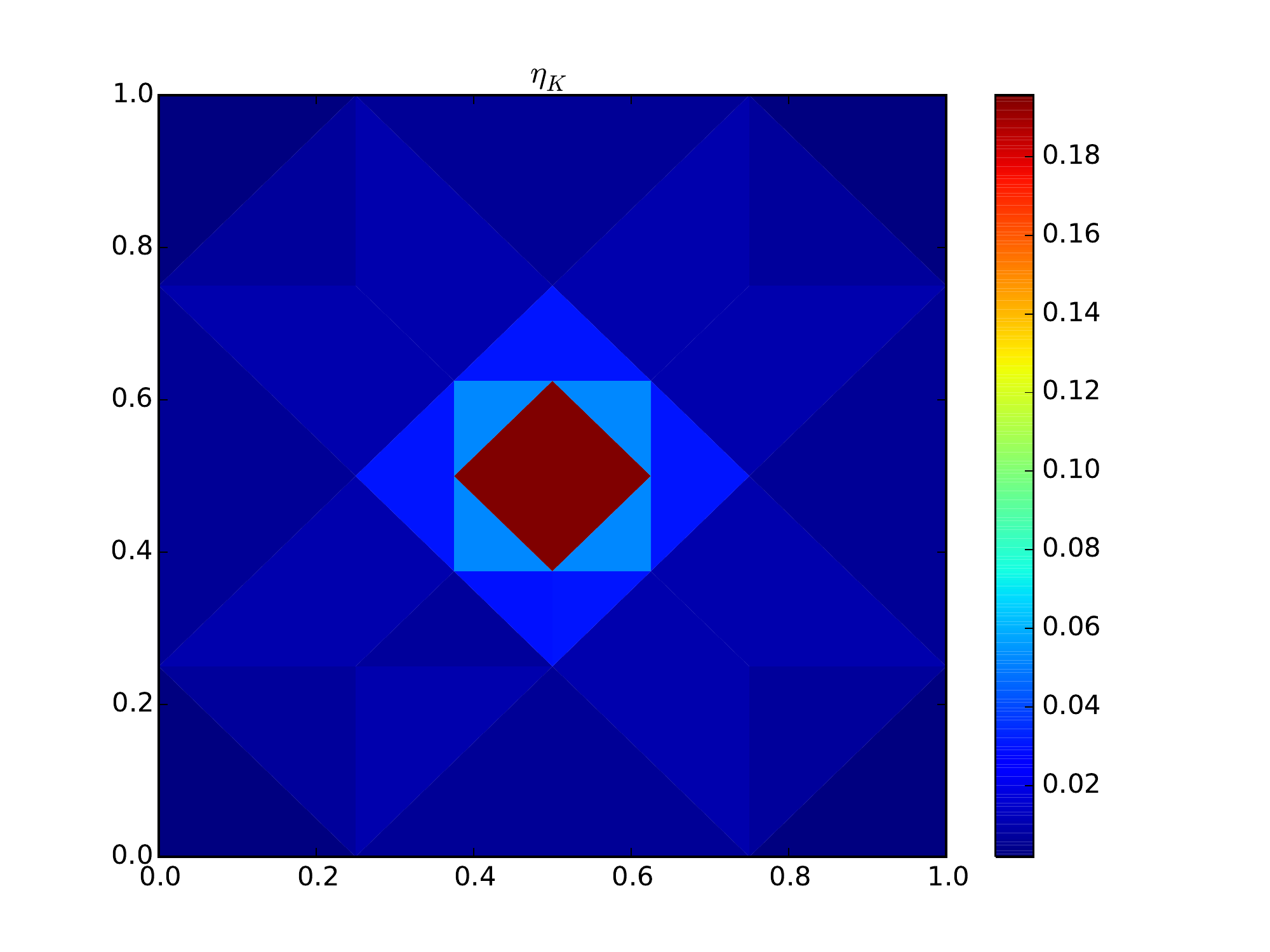}\\
    \includegraphics[width=0.49\textwidth]{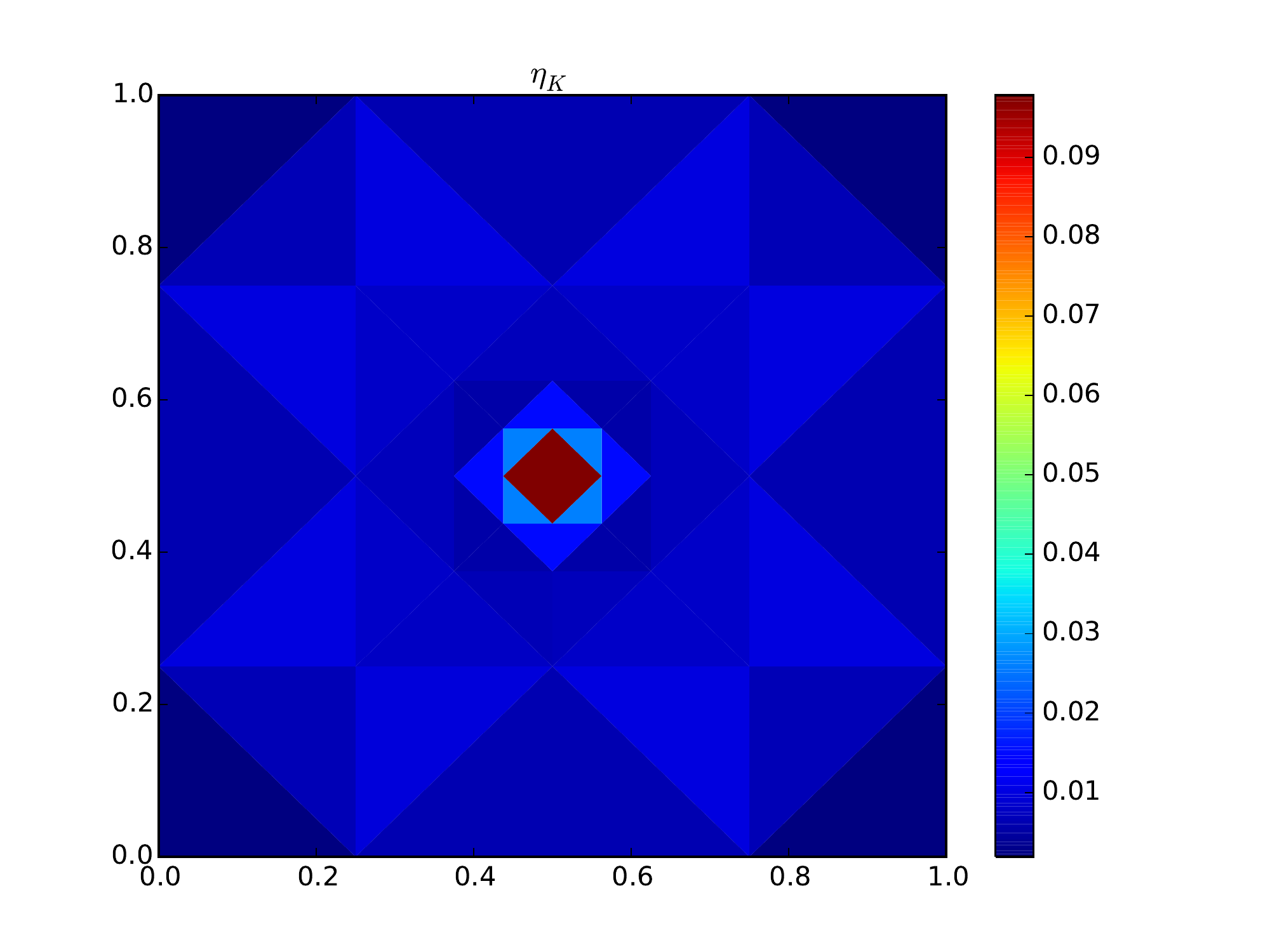}
    \includegraphics[width=0.49\textwidth]{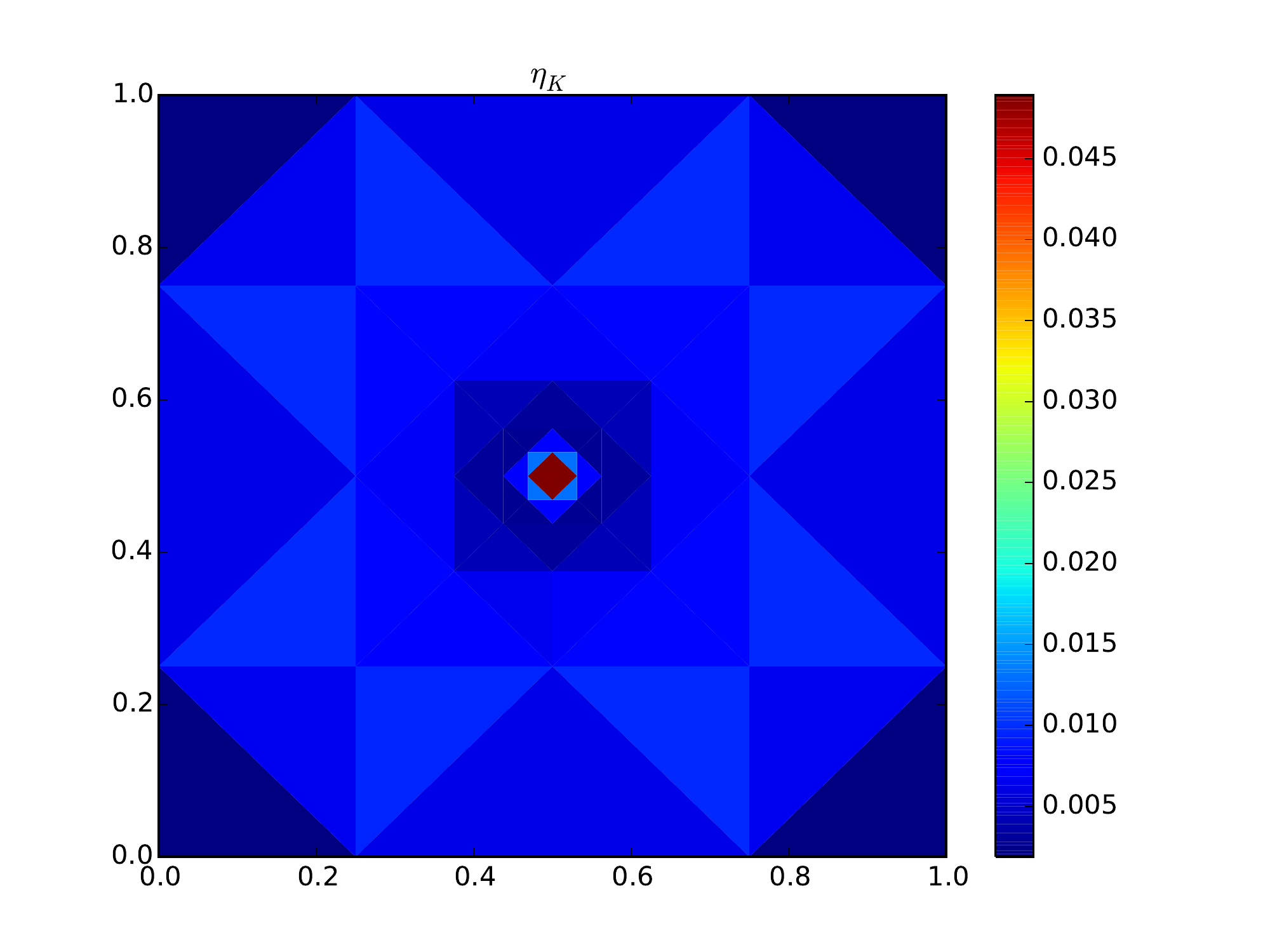}\\
    \includegraphics[width=0.49\textwidth]{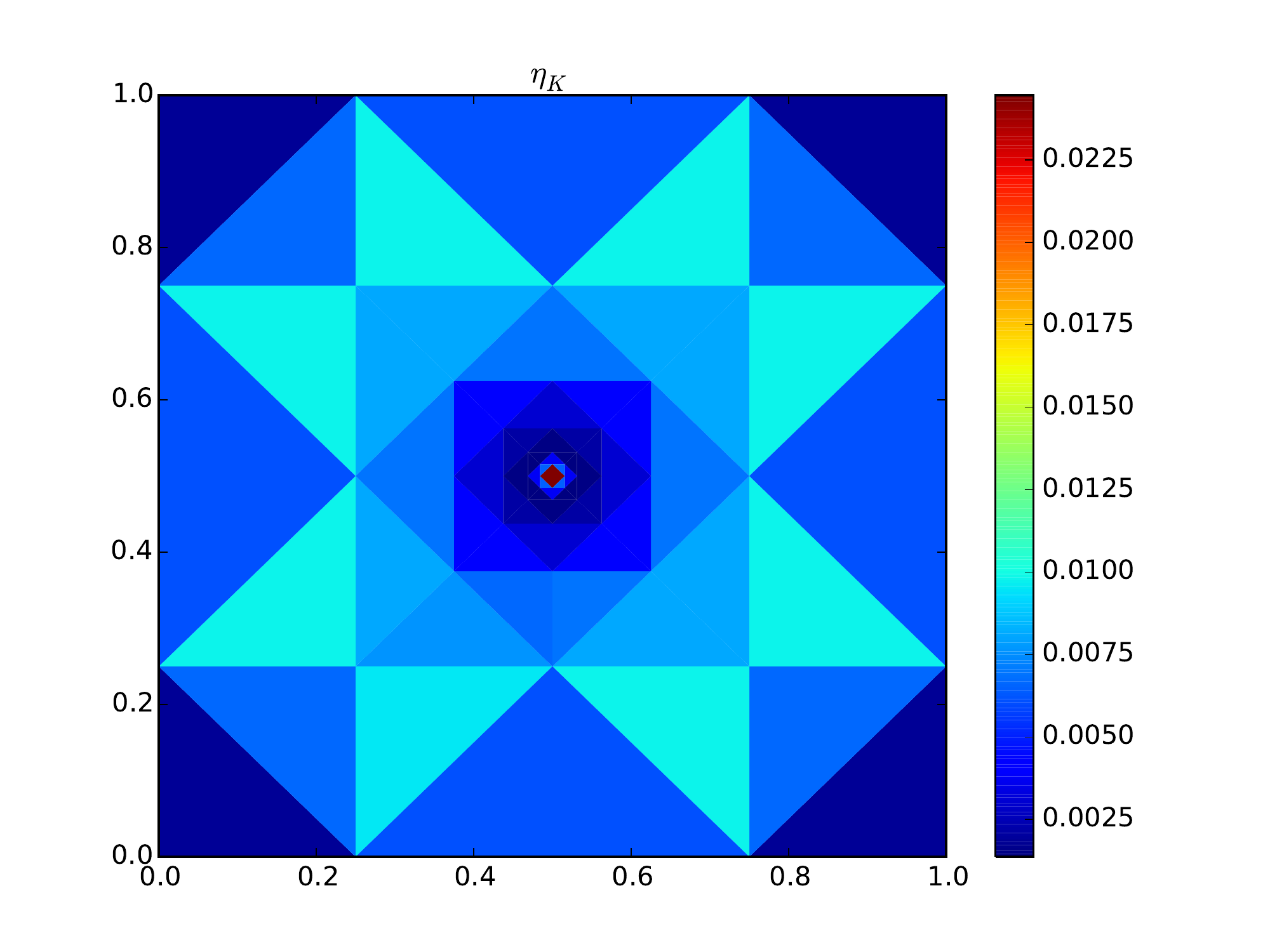}
    \includegraphics[width=0.49\textwidth]{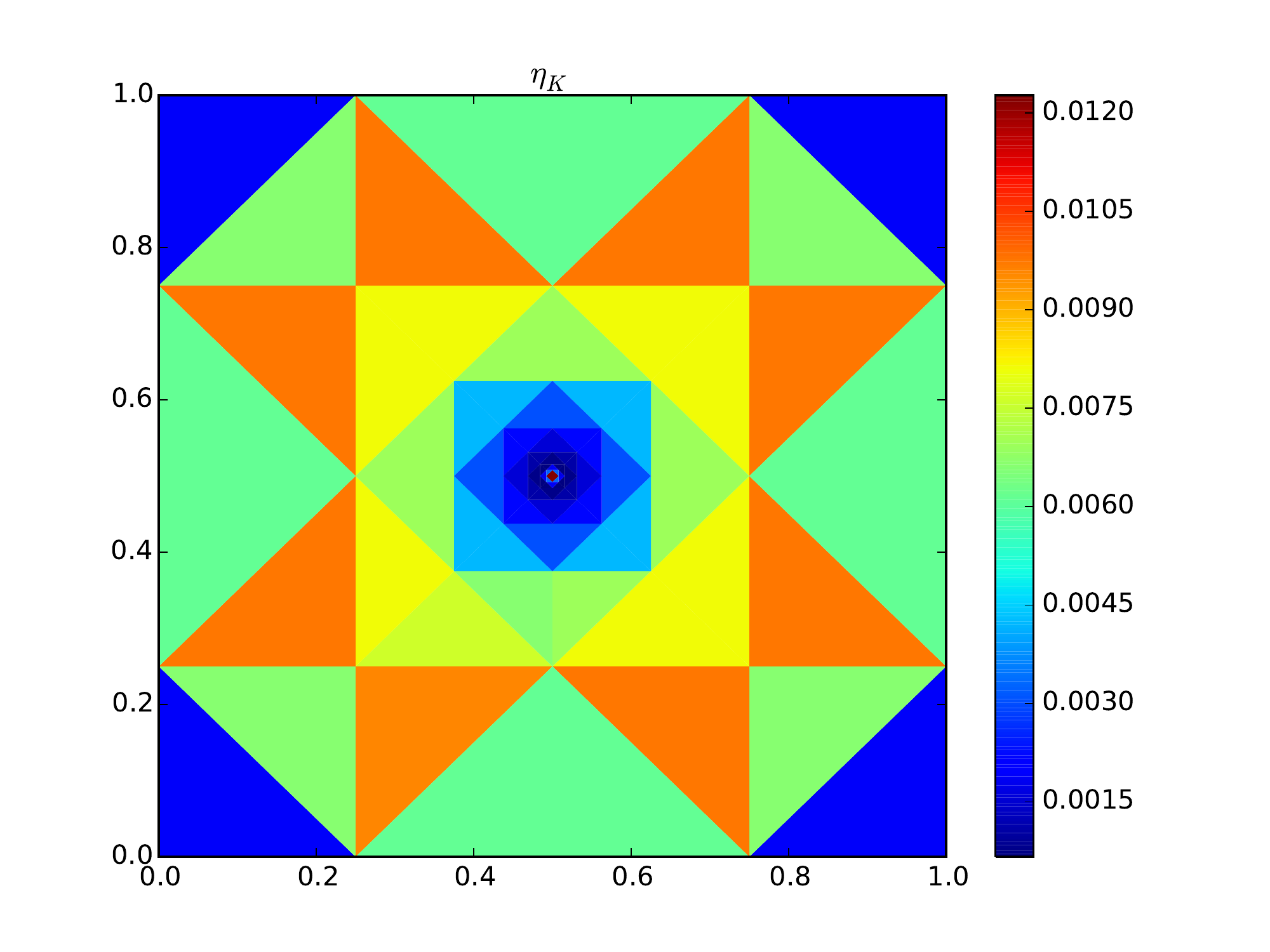}
    \caption{Elementwise error estimators in the point load case.}
    \label{fig:estim-pointload}
\end{figure}

\pgfplotstableread{
    ndofs nelems energynorm eta
    70 8 0.0334469831818 1.03051270004
    206 32 0.0169100336631 0.493682375884
    694 128 0.00843662159162 0.247183724801
    2534 512 0.00421783761299 0.123606218404
}\uniformpoint

\pgfplotstableread{
    ndofs nelems energynorm eta
    70 8 0.0334469831818 1.03051270004
    206 32 0.02365310132 0.82181176783
    278 48 0.0119938888183 0.416323779618
    350 64 0.00610693490316 0.212543389914
    422 80 0.0032646626221 0.114554845207
    494 96 0.00199970066259 0.0715111566944
    566 112 0.00139625006813 0.0557785320735
}\adaptivepoint


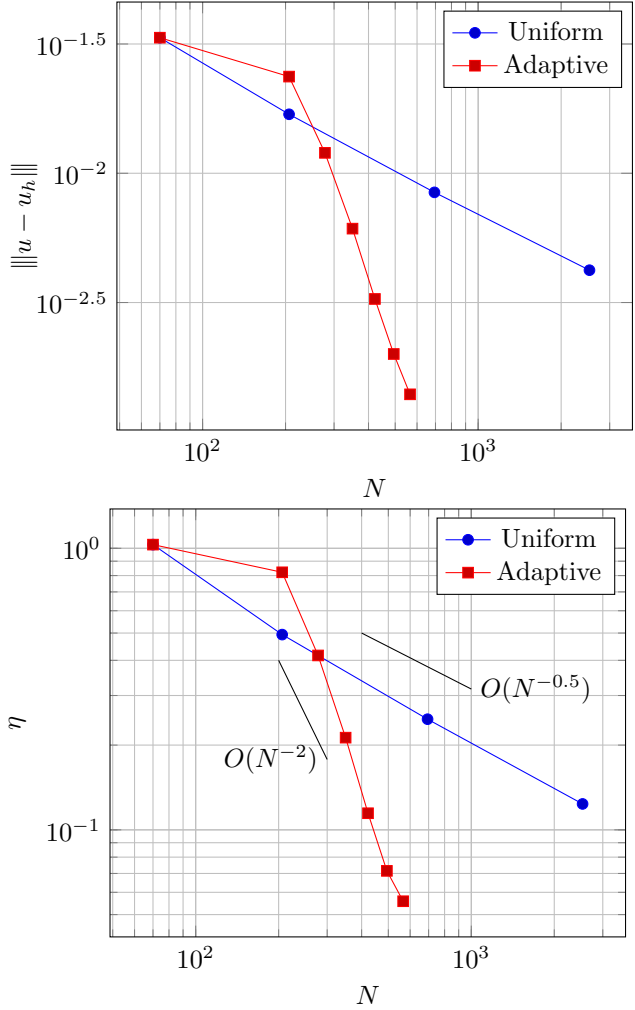
\begin{figure}
    \centering
    \begin{tikzpicture}
        \begin{axis}[
                xmode = log,
                ymode = log,
                xlabel = {$N$},
                ylabel = {$\vertiii{u-u_h}$},
                grid = both
            ]
            \addplot table[x=ndofs,y=energynorm] {\uniformpoint};
            \addplot table[x=ndofs,y=energynorm] {\adaptivepoint};
            \addlegendentry{Uniform}
            \addlegendentry{Adaptive}
        \end{axis}
    \end{tikzpicture}\\
    \begin{tikzpicture}
        \begin{axis}[
                xmode = log,
                ymode = log,
                xlabel = {$N$},
                ylabel = {$\eta$},
                grid = both
            ]
            \addplot table[x=ndofs,y=eta] {\uniformpoint};
            \addplot table[x=ndofs,y=eta] {\adaptivepoint};
            \addplot+ [black, domain=4e2:1e3, mark=none] {exp(-0.5*ln(x) + ln(5e-1) - (-0.5)*ln(4e2)))} node[right,pos=1.0]{$O(N^{-0.5})$};
            \addplot+ [black, domain=2e2:3e2, mark=none] {exp(-2*ln(x) + ln(4e-1) - (-2)*ln(2e2)))} node[left,pos=1.0]{$O(N^{-2})$};
            \addlegendentry{Uniform}
            \addlegendentry{Adaptive}
        \end{axis}
    \end{tikzpicture}
    \caption{The results of the point load case.}
    \label{fig:pointloadgraph}
\end{figure}

\begin{figure}
    \centering
    \begin{tikzpicture}
        \begin{axis}[
                xlabel = {$N$},
                ylabel = {$c^{-1} \frac{\eta}{\vertiii{u-u_h}}$},
                grid = both
            ]
            \addplot table[x=ndofs,y expr=\thisrow{eta}/\thisrow{energynorm}/35] {\adaptivepoint};
        \end{axis}
    \end{tikzpicture}
    \caption{The efficiency of the estimator in the point load case. The
    normalization parameter $c$ is chosen as the mean value of the ratio
    $\eta/\vertiii{u-u_h}$.}
    \label{fig:pointloadefficiency}
\end{figure}
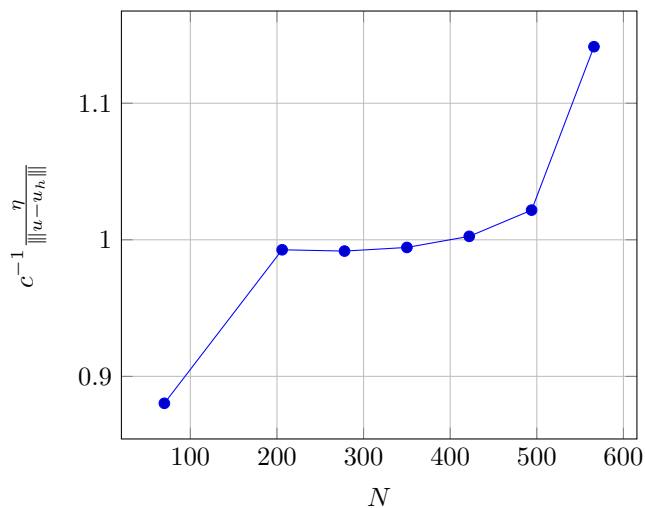

\begin{figure}
    \includegraphics[width=0.49\textwidth]{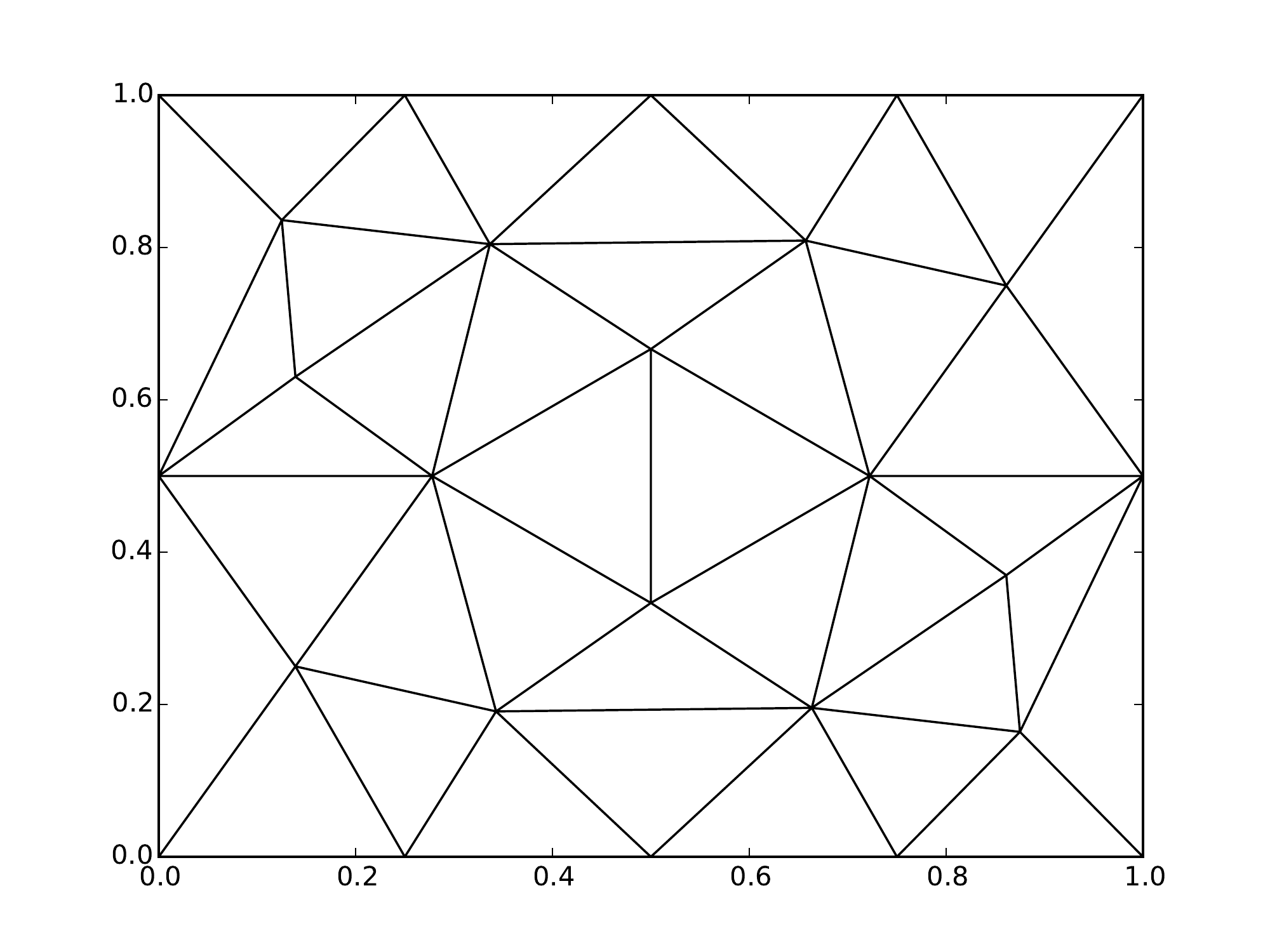}
    \includegraphics[width=0.49\textwidth]{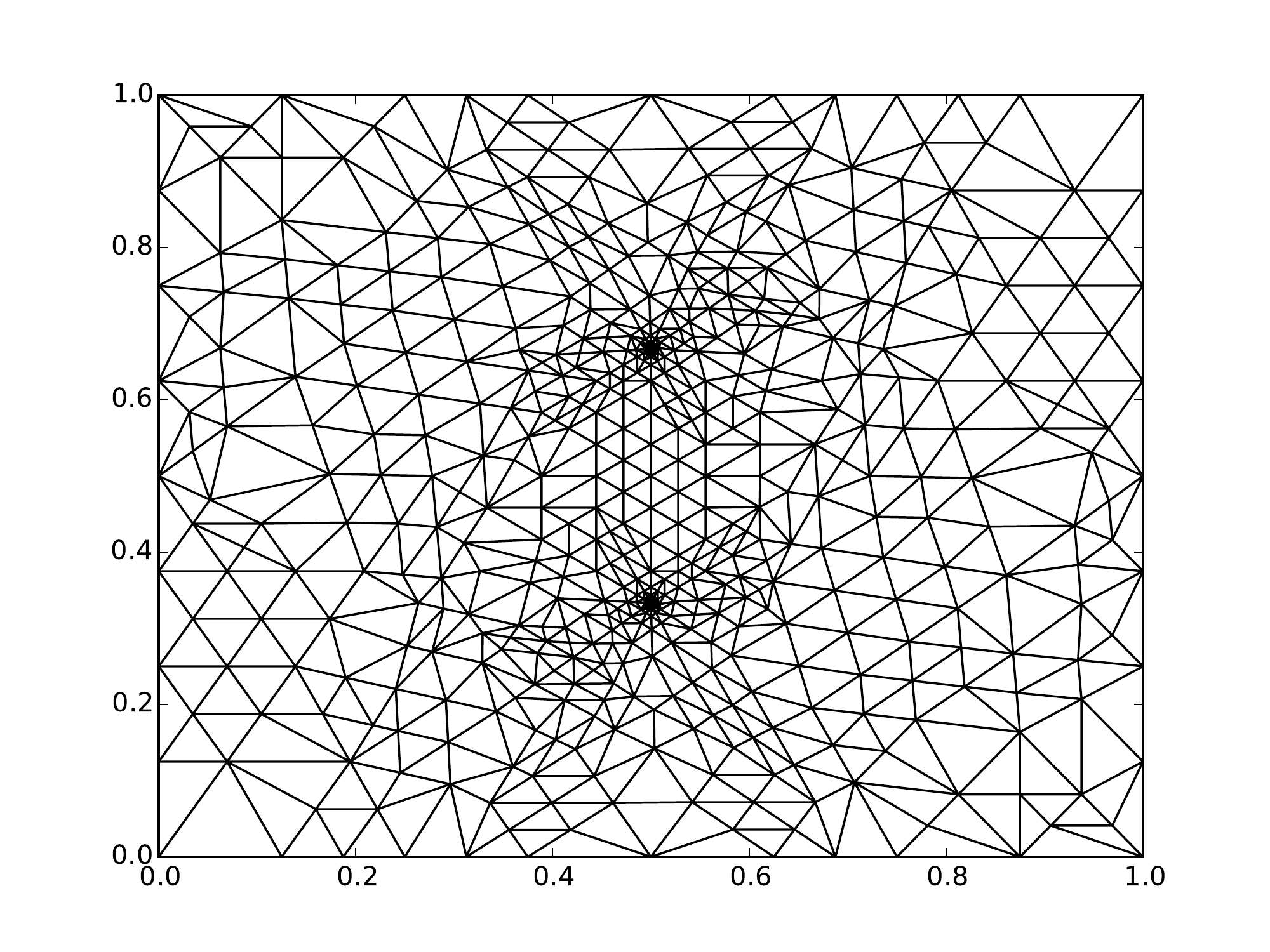}
    \caption{Initial and 6 times refined meshes for the line load case.}
    \label{fig:mesh-lineload}
\end{figure}

\begin{figure}
    \includegraphics[width=0.49\textwidth]{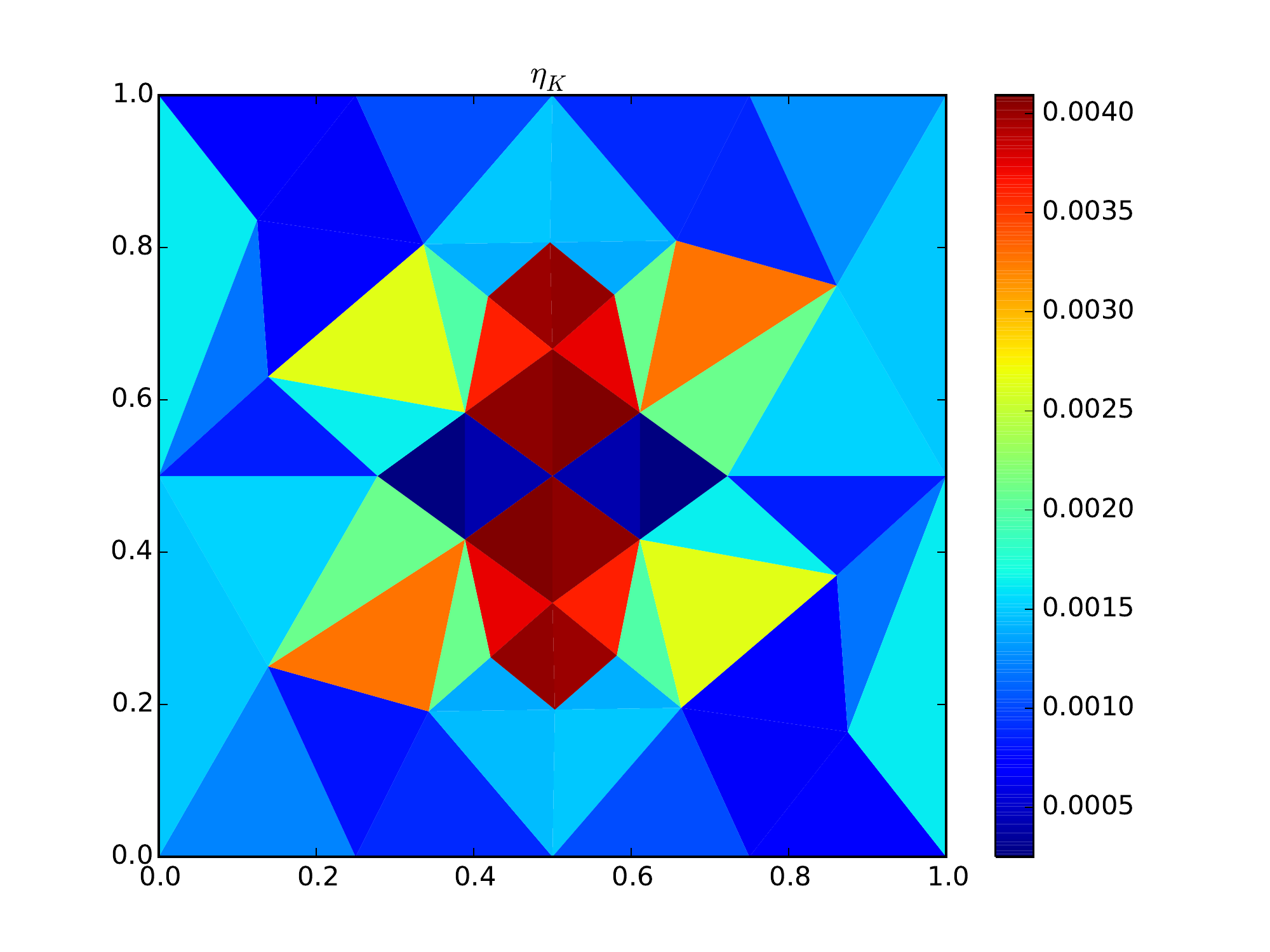}
    \includegraphics[width=0.49\textwidth]{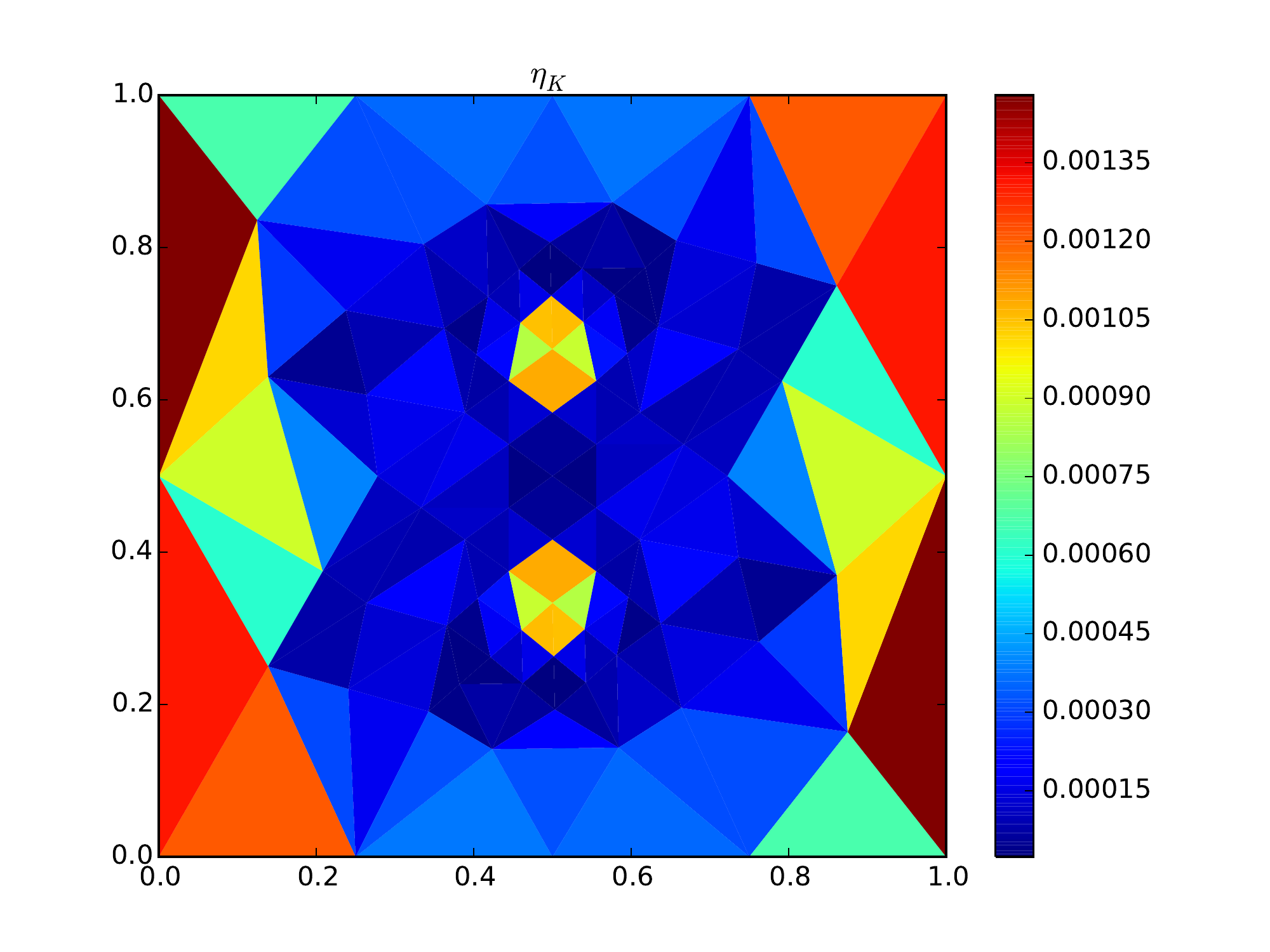}\\
    \includegraphics[width=0.49\textwidth]{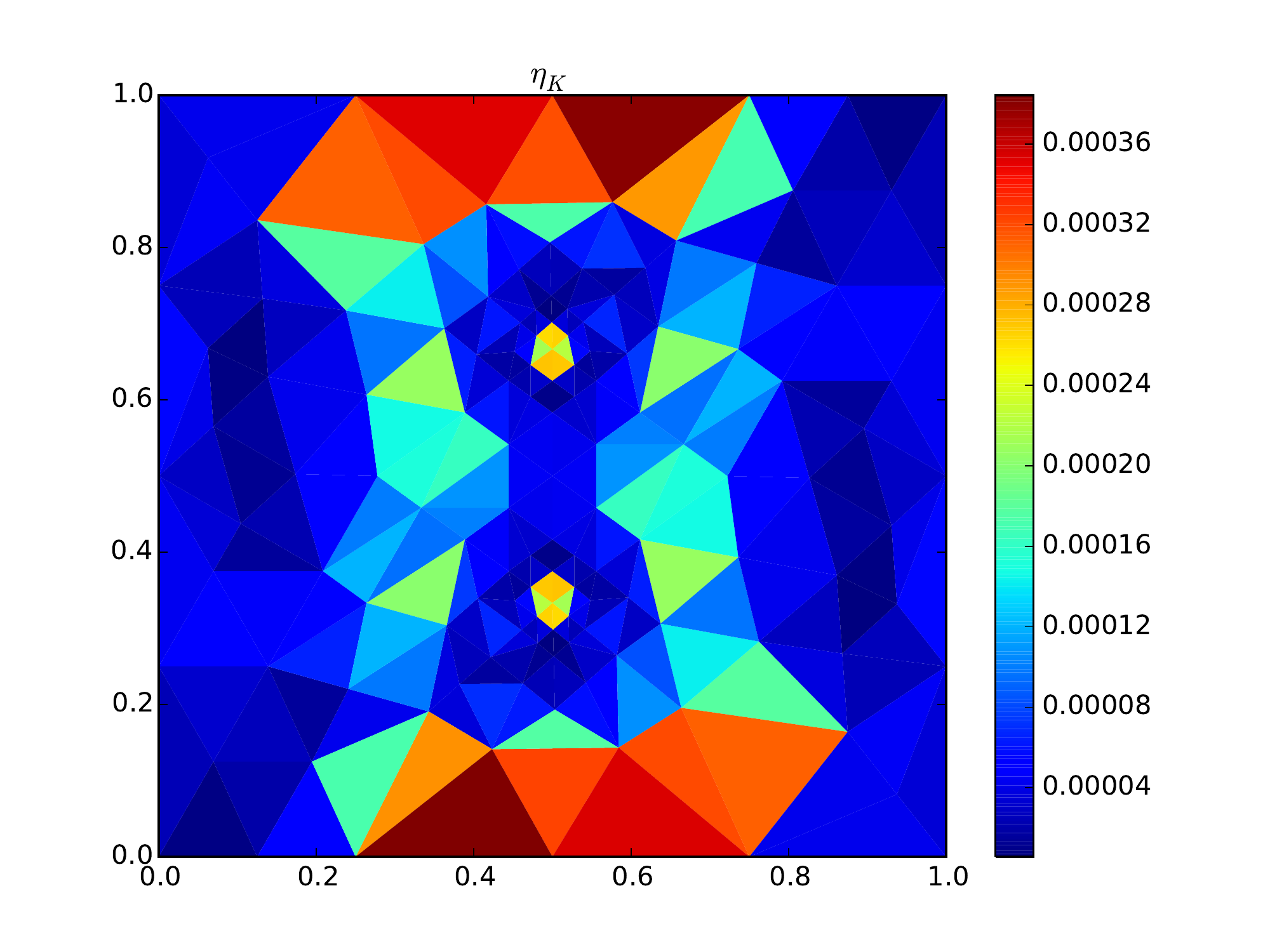}
    \includegraphics[width=0.49\textwidth]{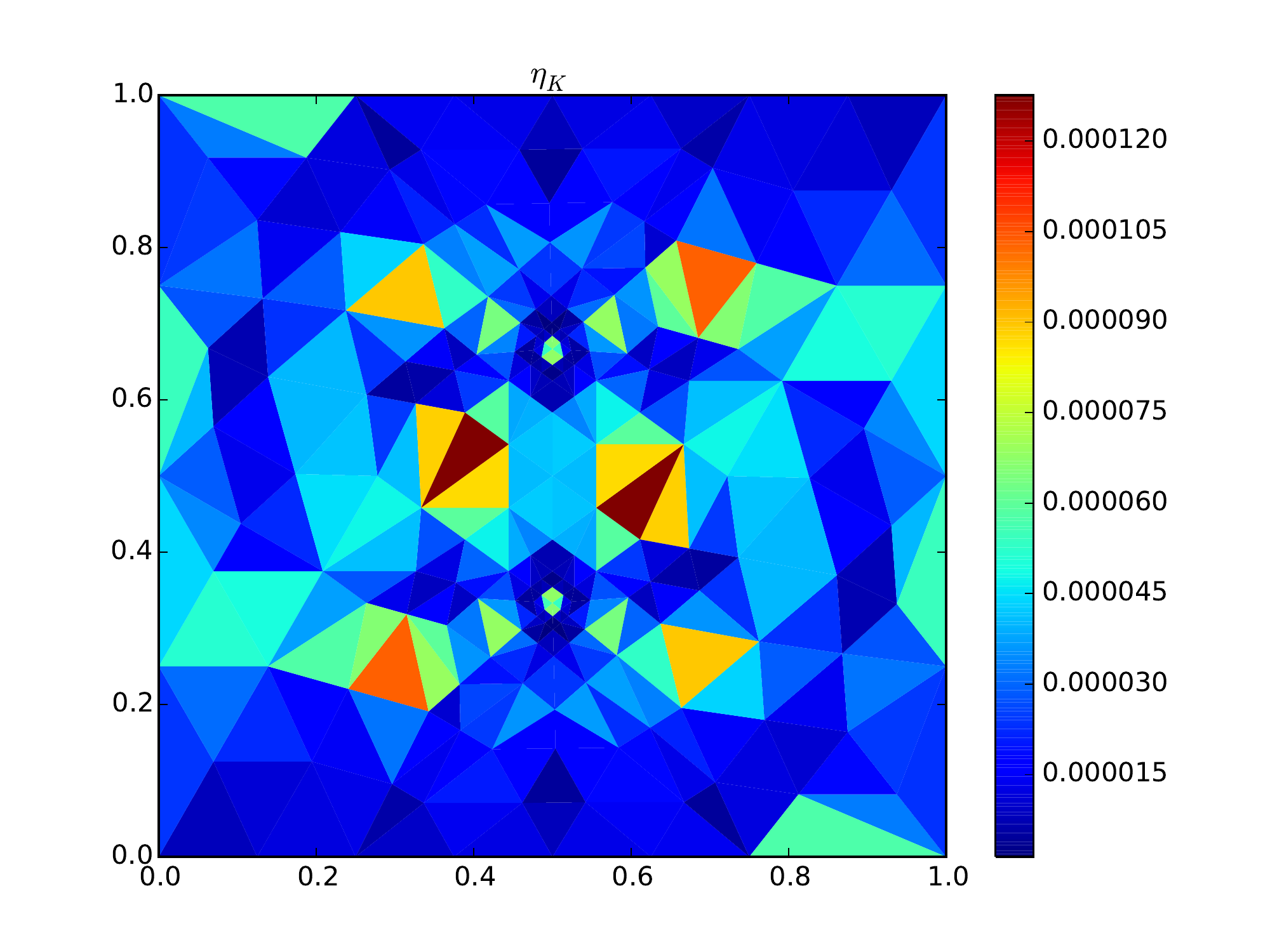}\\
    \includegraphics[width=0.49\textwidth]{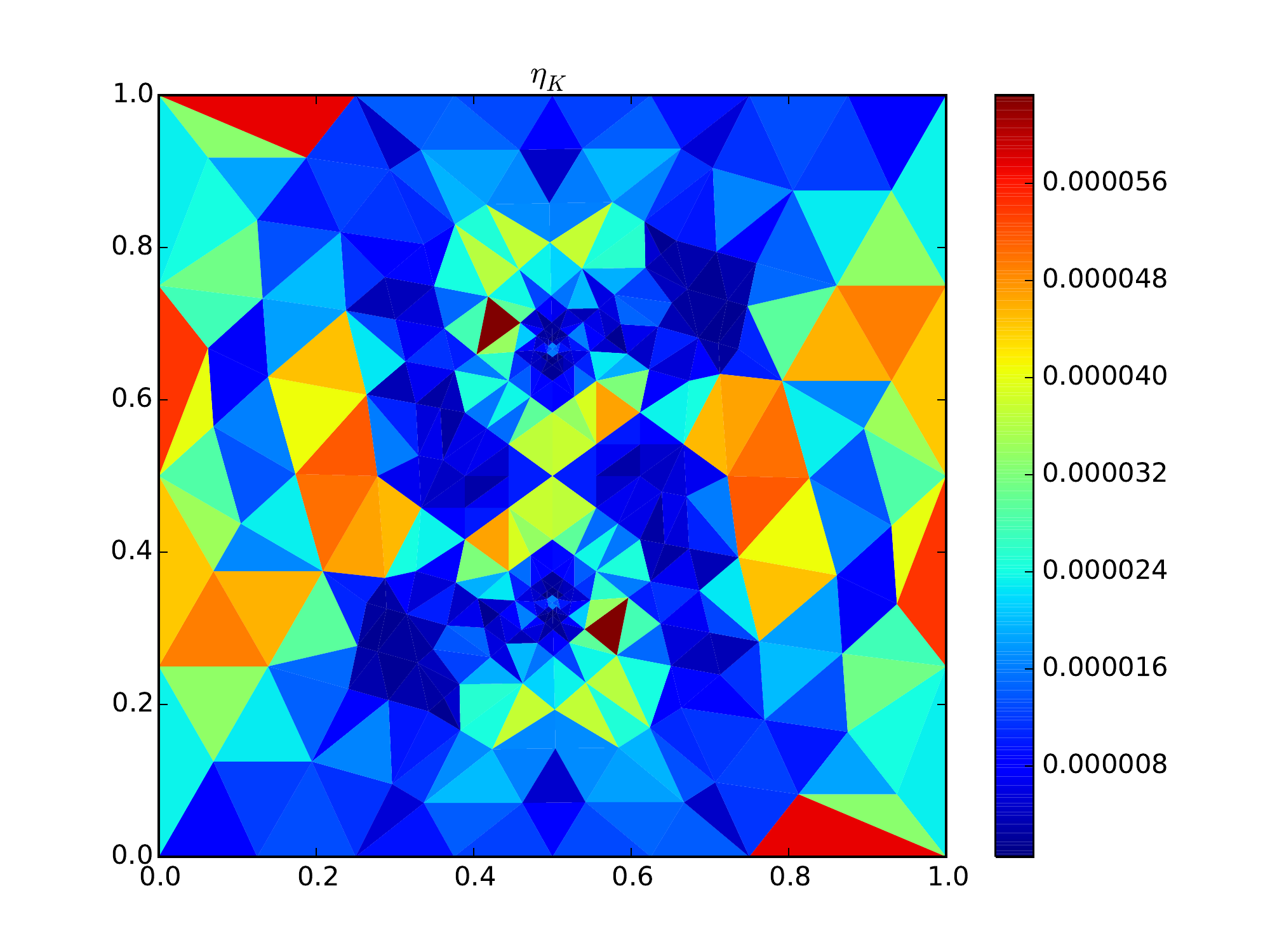}
    \includegraphics[width=0.49\textwidth]{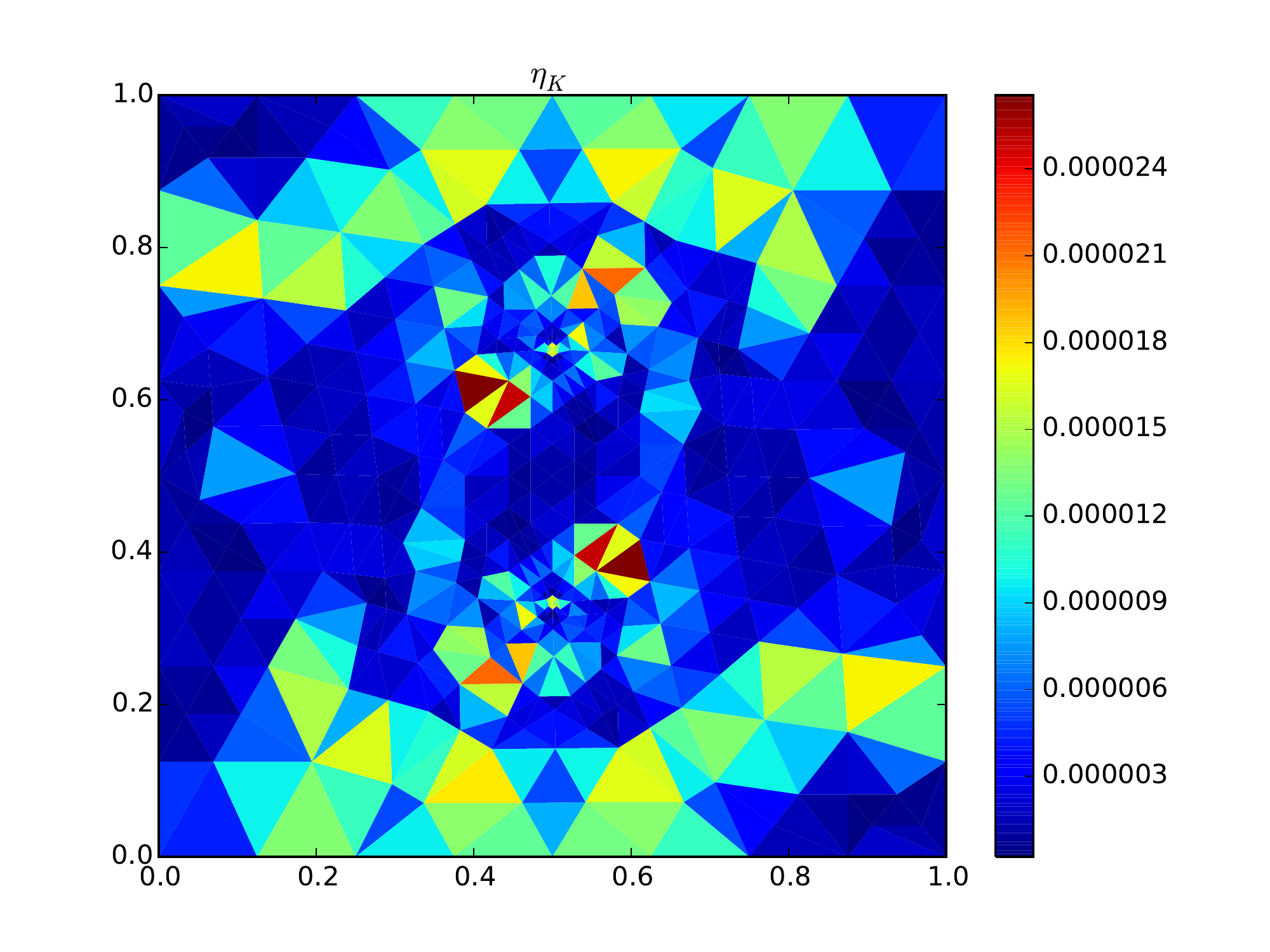}
    \caption{Elementwise error estimators for the line load case.}
    \label{fig:estim-lineload}
\end{figure}

\pgfplotstableread{
    ndofs nelems eta
    219 38 0.112263251802
    774 152 0.0392580419907
    2910 608 0.0138990740805
    11286 2432 0.00491278489704
}\uniformline

\pgfplotstableread{
    ndofs nelems eta
    219 38 0.0874824228262
    318 60 0.0172422350444
    732 152 0.00563138145735
    1158 242 0.00175456884437
    1712 362 0.000633039818195
    2270 486 0.00041008970146
    3394 728 0.000183071024098
}\adaptiveline


\begin{figure}
    \centering
    \begin{tikzpicture}
        \begin{axis}[
                xmode = log,
                ymode = log,
                xlabel = {$N$},
                ylabel = {$\eta$},
                grid = both
            ]
            \addplot table[x=ndofs,y=eta] {\uniformline};
            \addplot table[x=ndofs,y=eta] {\adaptiveline};
            \addplot+ [black, domain=1e3:9e3, mark=none] {exp(-3/4*ln(x) + ln(2e-2) - (-3/4)*ln(1e3)))} node[below,pos=1.0]{$O(N^{-0.75})$};
            \addplot+ [black, domain=3e2:1e3, mark=none] {exp(-2*ln(x) + ln(1e-2) - (-2)*ln(3e2)))} node[below,pos=1.0]{$O(N^{-2})$};
            \addlegendentry{Uniform}
            \addlegendentry{Adaptive}
        \end{axis}
    \end{tikzpicture}
    \caption{Results of the line load case.}
    \label{fig:lineloadgraph}
\end{figure}
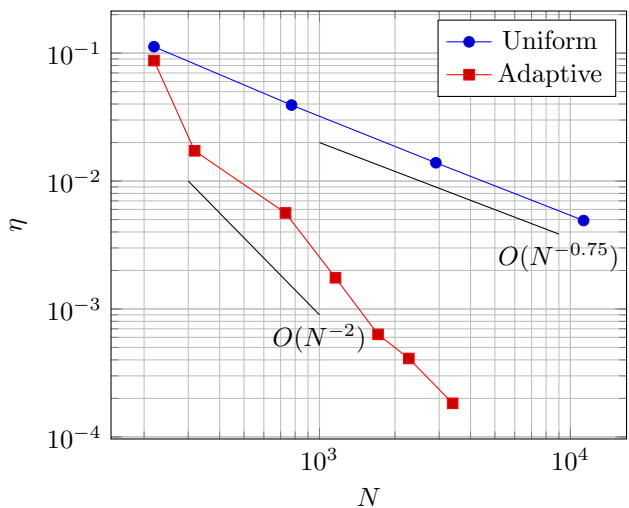

\pgfplotstableread{
    ndofs nelems eta
    208 34 0.0380160176302
    716 136 0.0080971107166
    2650 544 0.00102523606838
    10190 2176 0.000137845398212
}\uniformbox

\pgfplotstableread{
    ndofs nelems eta
    208 34 0.0380160176302
    407 72 0.00957105089795
    452 82 0.00538652794625
    653 122 0.00227568555901
    957 188 0.00119173613498
    1453 292 0.0004456239461
    1478 296 0.000394849345563
    1907 389 0.000227748630462
    2867 593 9.98047955912e-05
    3074 639 7.70829141819e-05
}\adaptivebox

\begin{figure}
    \includegraphics[width=0.49\textwidth]{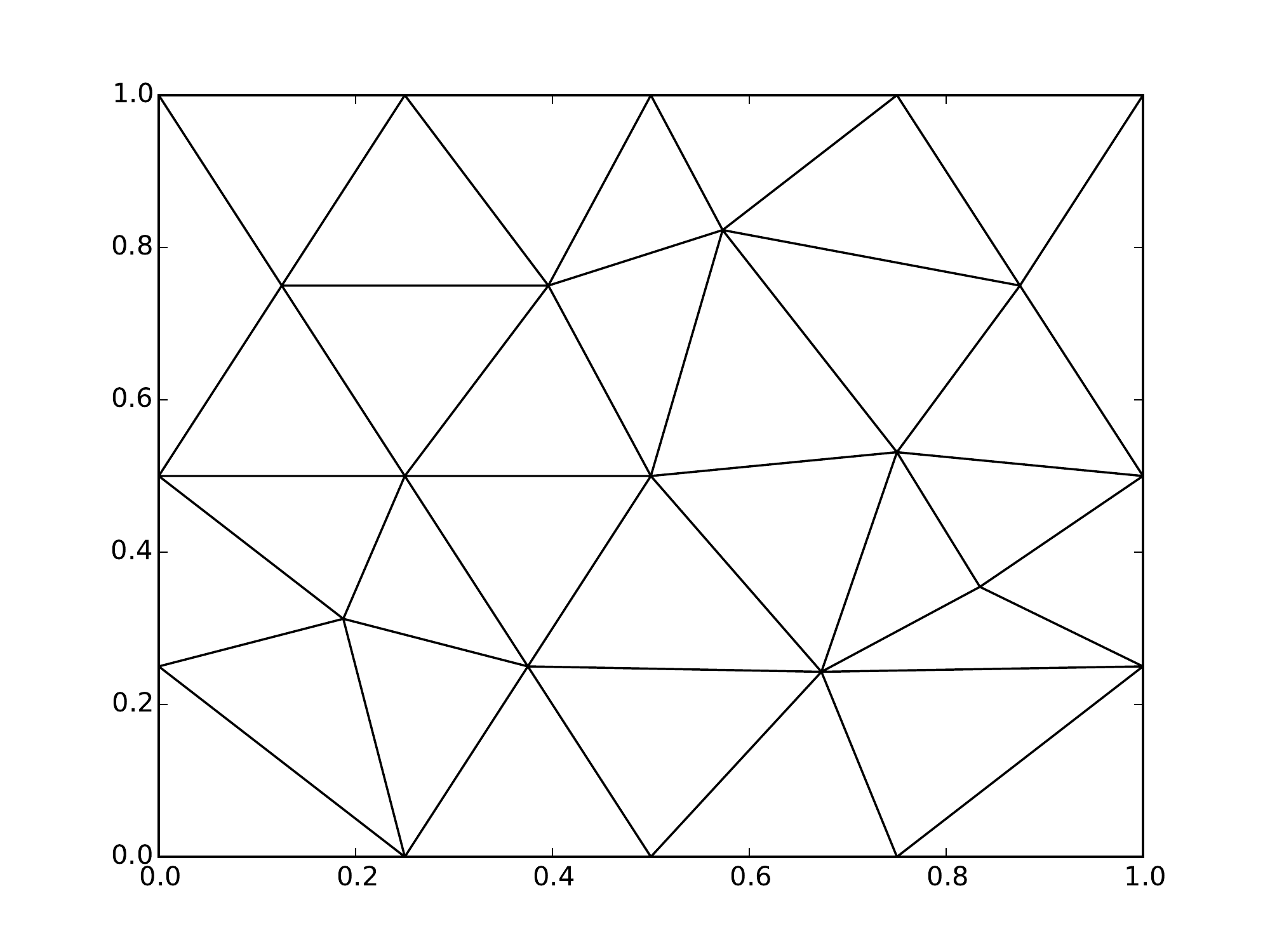}
    \includegraphics[width=0.49\textwidth]{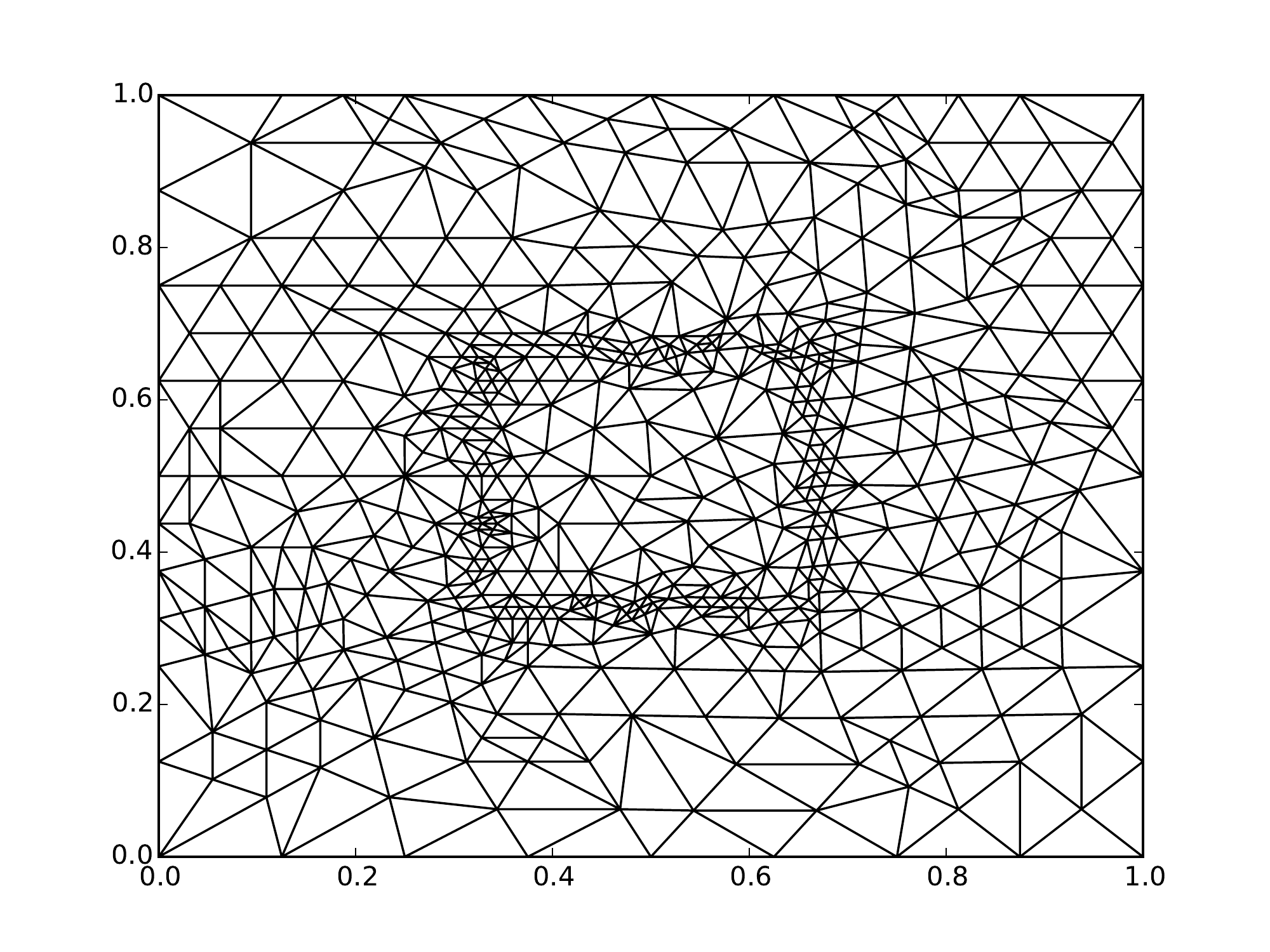}
    \caption{Initial and 8 times refined meshes for the square load case.}
    \label{fig:mesh-squareload}
\end{figure}


\begin{figure}
    \includegraphics[width=0.49\textwidth]{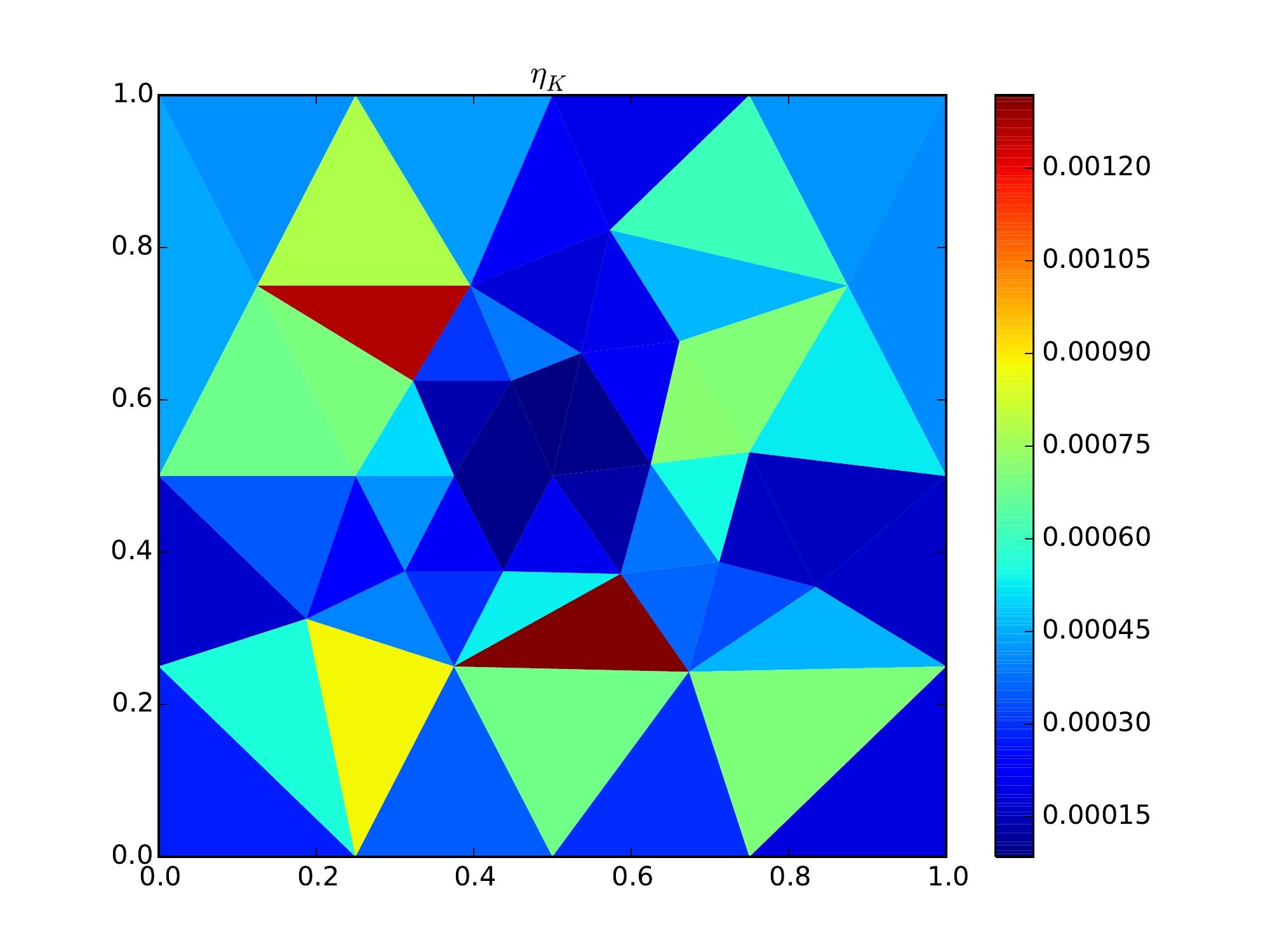}
    \includegraphics[width=0.49\textwidth]{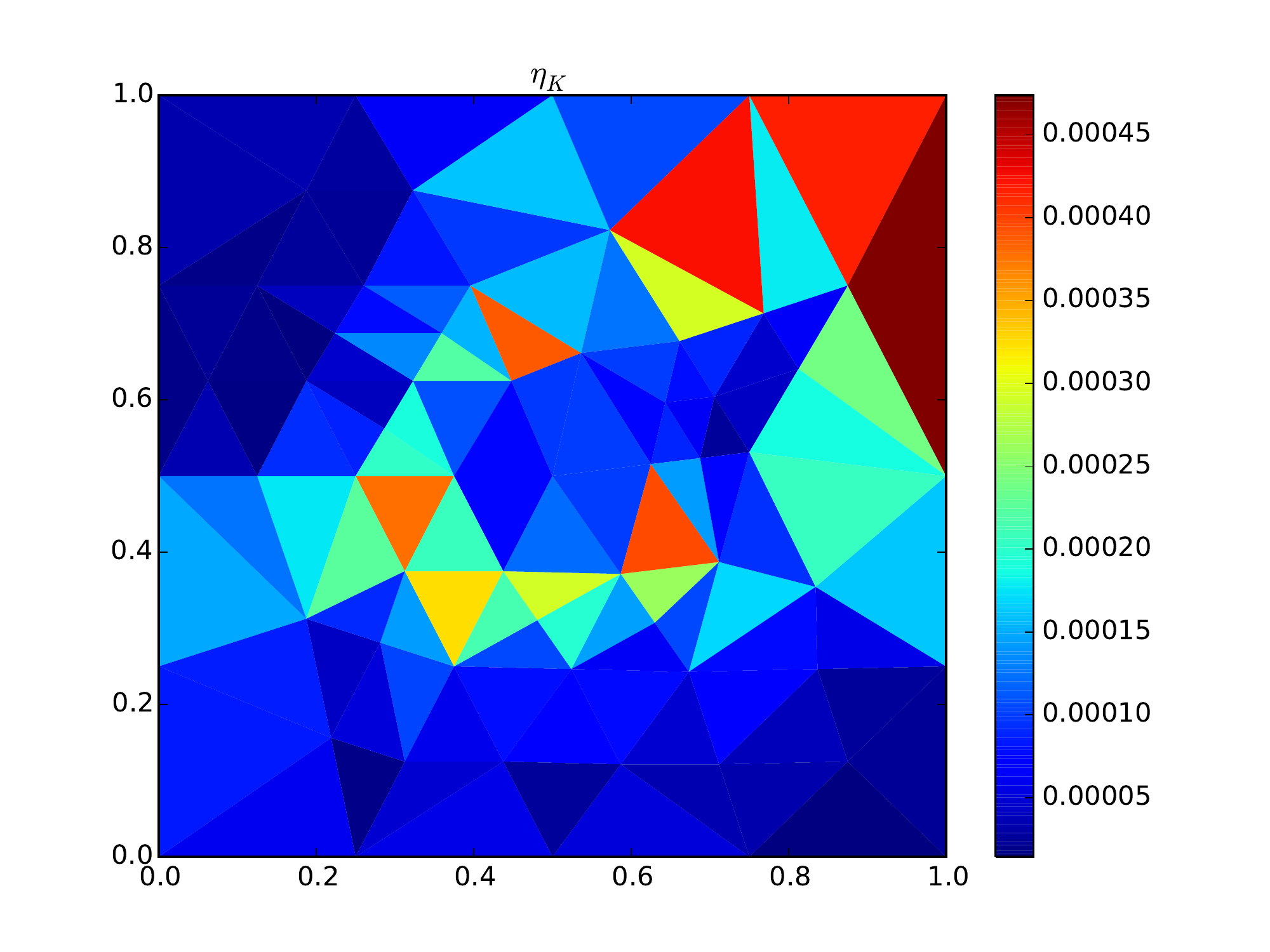}\\
    \includegraphics[width=0.49\textwidth]{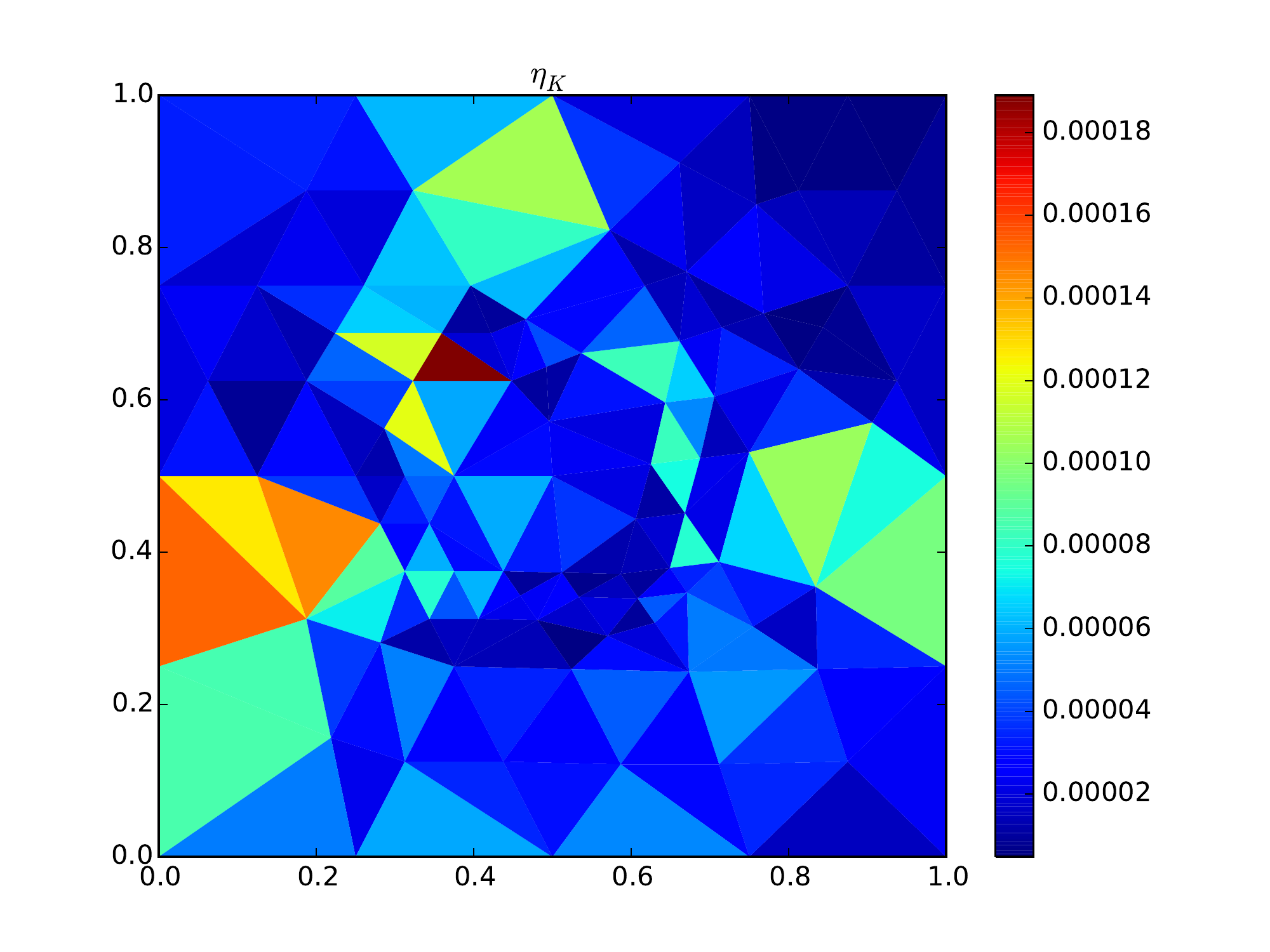}
    \includegraphics[width=0.49\textwidth]{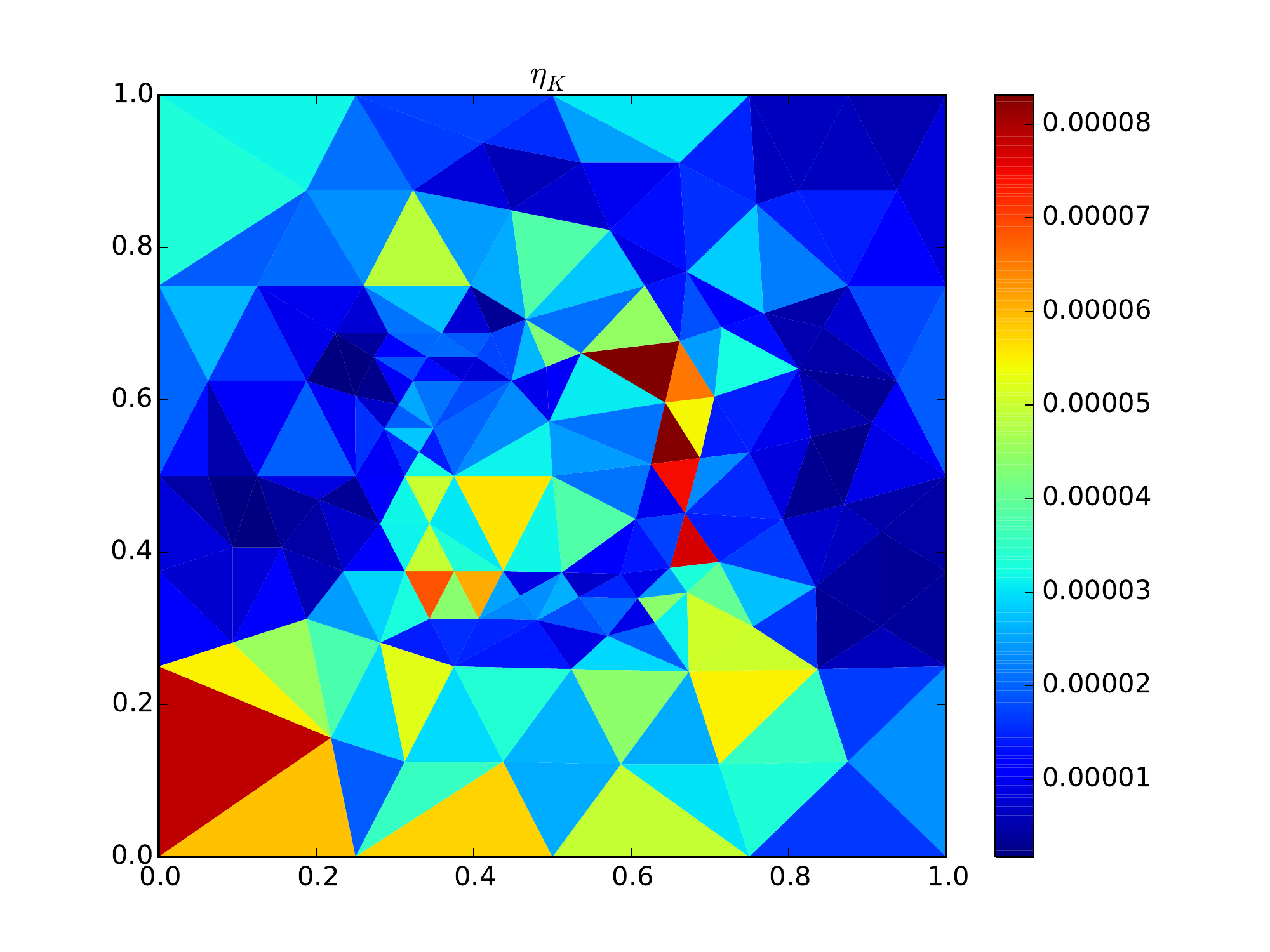}\\
    \includegraphics[width=0.49\textwidth]{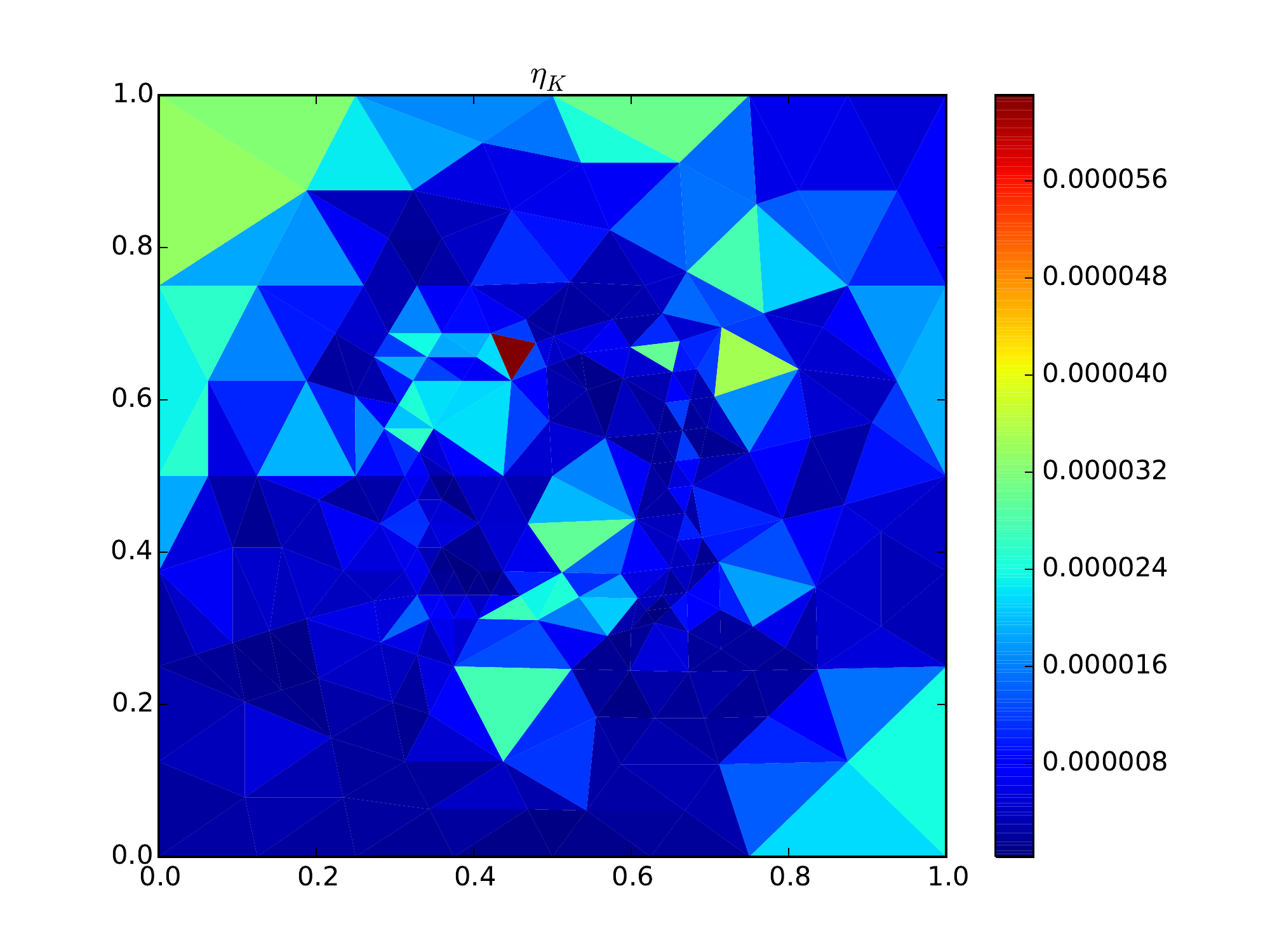}
    \includegraphics[width=0.49\textwidth]{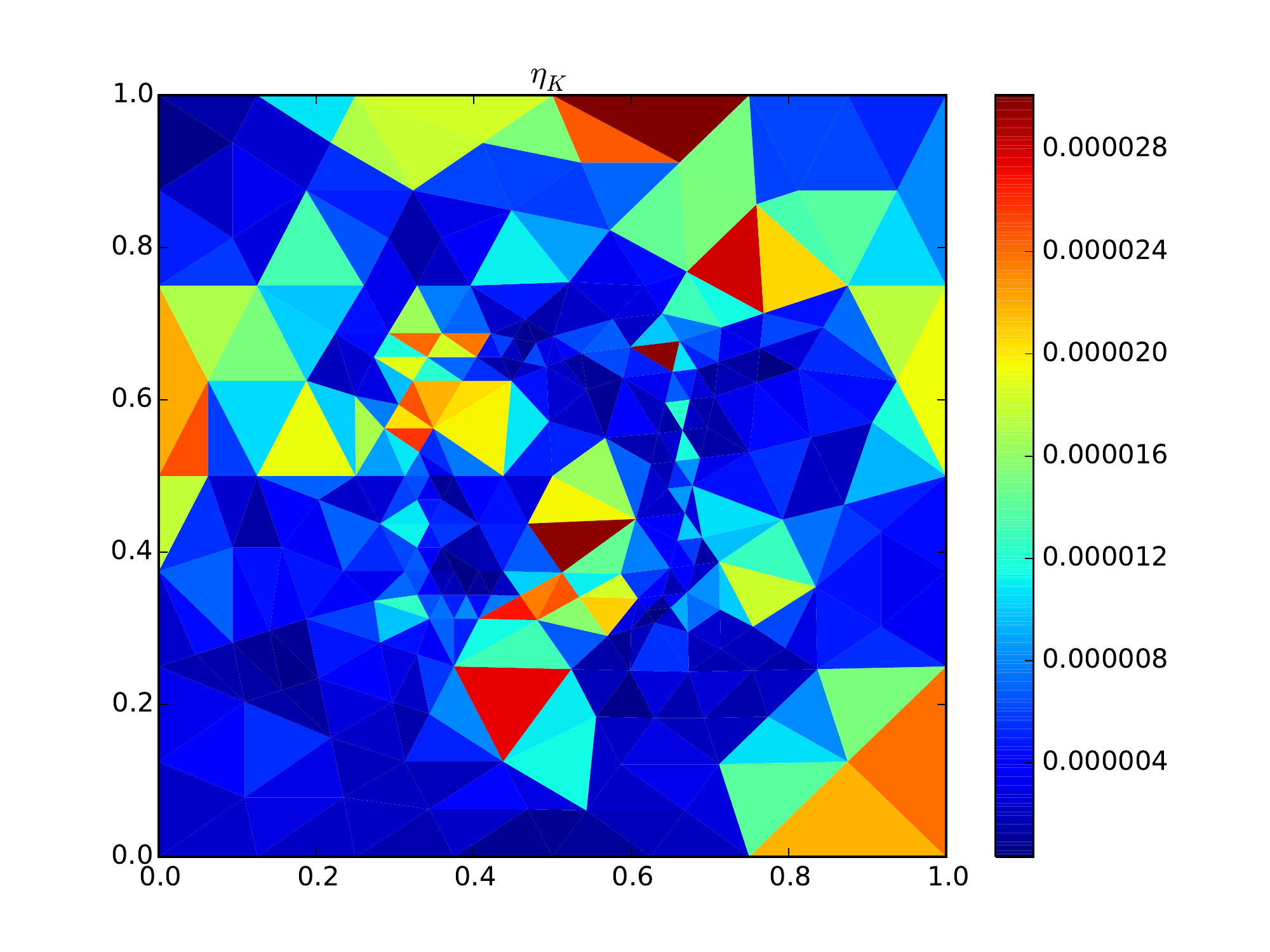}\\
    \includegraphics[width=0.49\textwidth]{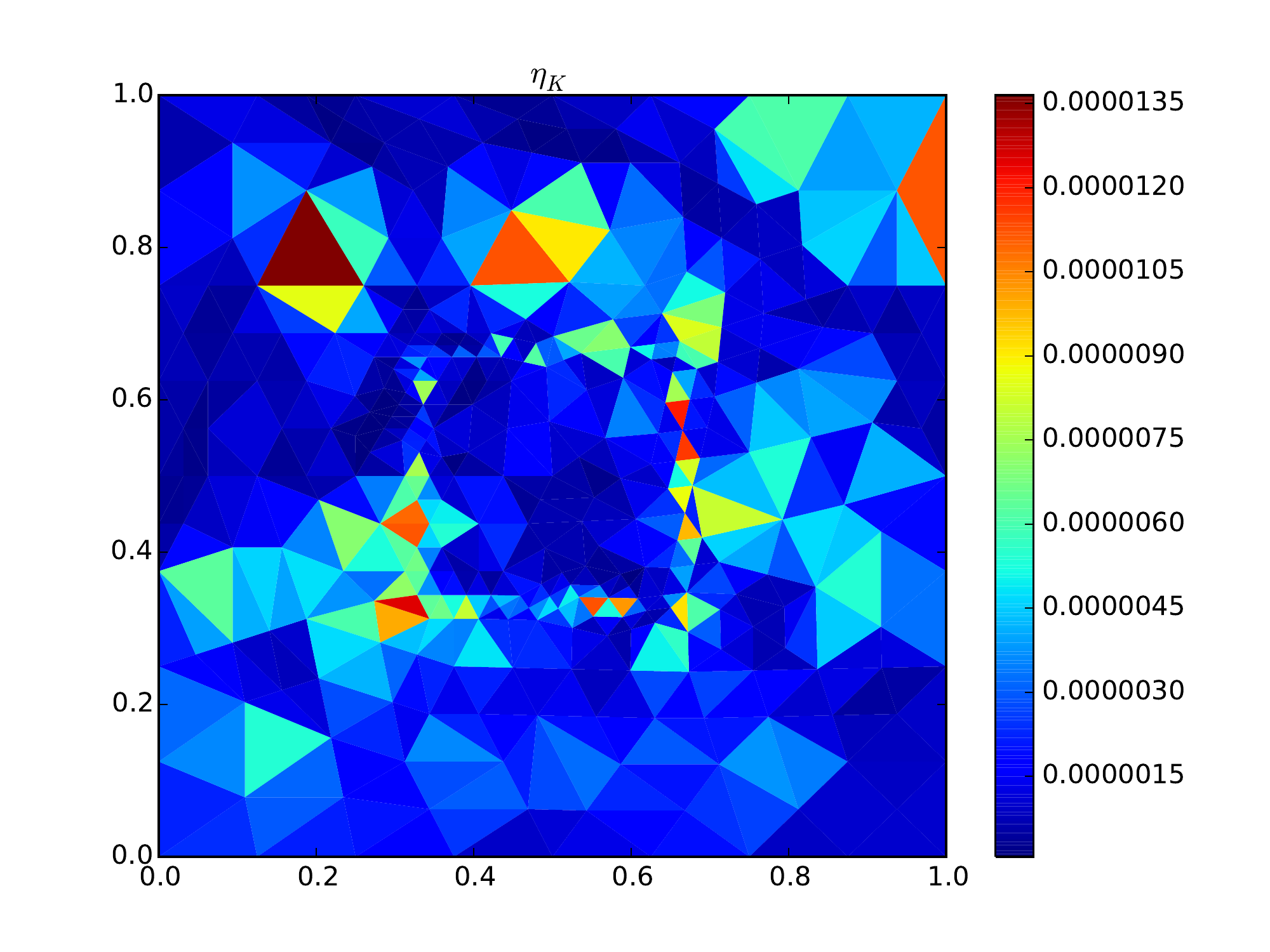}
    \includegraphics[width=0.49\textwidth]{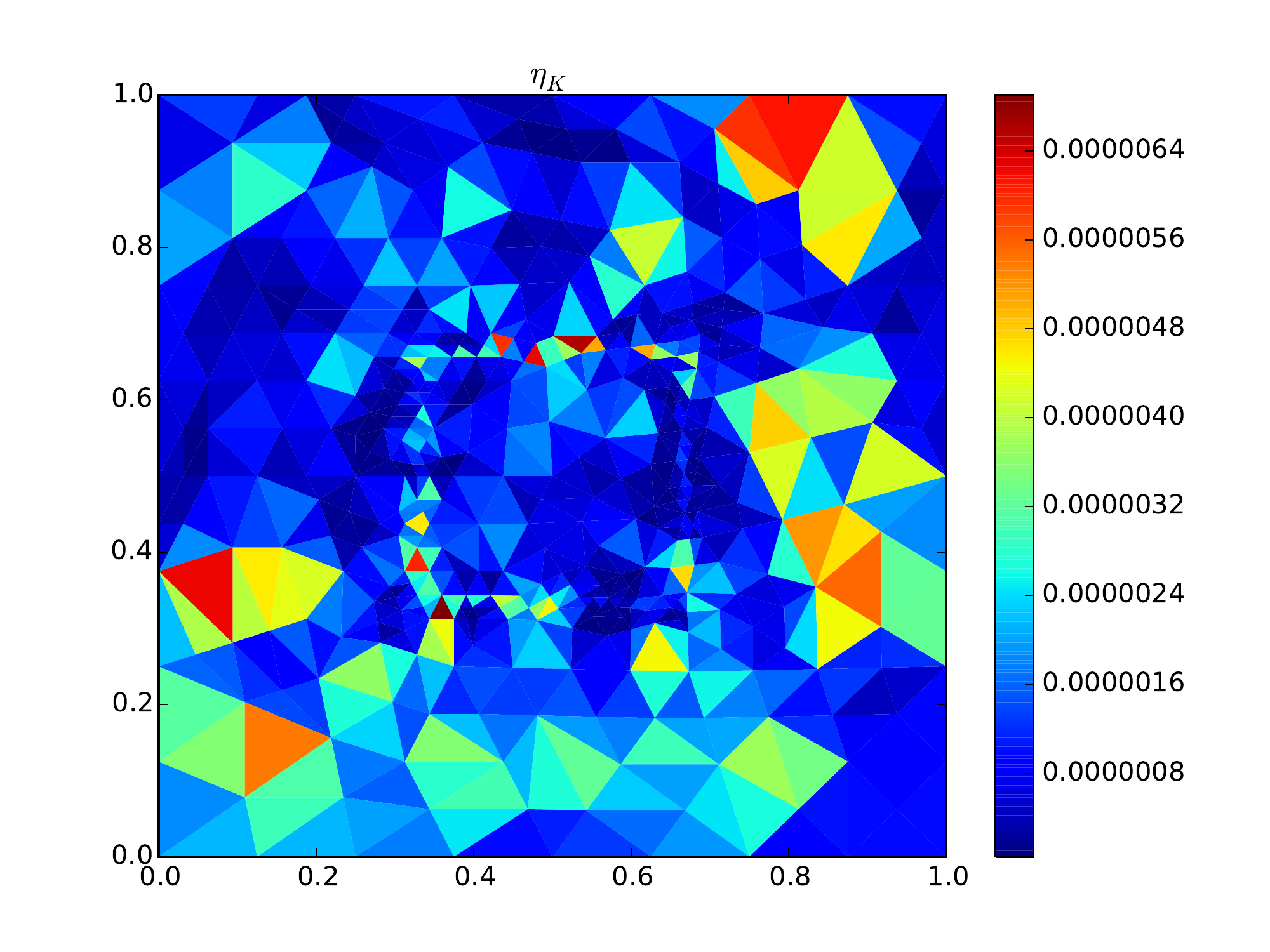}
    \caption{Elementwise error estimators in the square load case.}
    \label{fig:estim-squareload}
\end{figure}

\begin{figure}
    \centering
    \begin{tikzpicture}
        \begin{axis}[
                xmode = log,
                ymode = log,
                xlabel = {$N$},
                ylabel = {$\eta$},
                grid = both
            ]
            \addplot table[x=ndofs,y=eta] {\uniformbox};
            \addplot table[x=ndofs,y=eta] {\adaptivebox};
            \addplot+ [black, domain=1e3:7e3, mark=none] {exp(-2.5/2.0*ln(x) + ln(2e-2) - (-2.5/2.0)*ln(1e3)))} node[below,pos=1.0]{$O(N^{-1.25})$};
            \addplot+ [black, domain=3e2:1e3, mark=none] {exp(-2.0*ln(x) + ln(4e-3) - (-2)*ln(3e2)))} node[below,pos=1.0]{$O(N^{-2})$};
            \addlegendentry{Uniform}
            \addlegendentry{Adaptive}
        \end{axis}
    \end{tikzpicture}
    \caption{Results of the square load case.}
    \label{fig:squareloadgraph}
\end{figure}
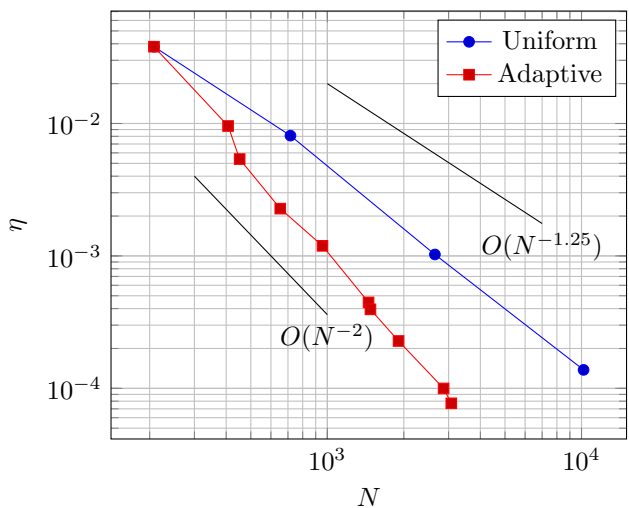
 
\subsection{L-shaped domain}

Next we solve the Kirchhoff plate problem in L-shaped domain with uniform loading $f=1$
and the following three sets of boundary conditions:
\begin{enumerate}
    \item Simply supported on all boundaries.
    \item Clamped on all boundaries.
    \item Free on the edges sharing the re-entrant corner and simply supported
        along the rest of the boundary.
\end{enumerate}
Due to the presence of a re-entrant corner, the solutions belong to
$H^{2.33}(\Omega)$, $H^{2.54}(\Omega)$ and $H^{2.64}(\Omega)$ in the cases 1,
2 and 3, respectively(see~\cite{MR}). As before, we use fifth-order Argyris
elements to demonstrate the effectiveness of the adaptive solution strategy.
The initial and the final meshes are shown in Fig.~\ref{fig:mesh-lshaped}. The resulting
total error estimators and unknown counts are visualized in Fig.~\ref{fig:lshapedgraph}.

\begin{figure}
    \includegraphics[width=0.49\textwidth]{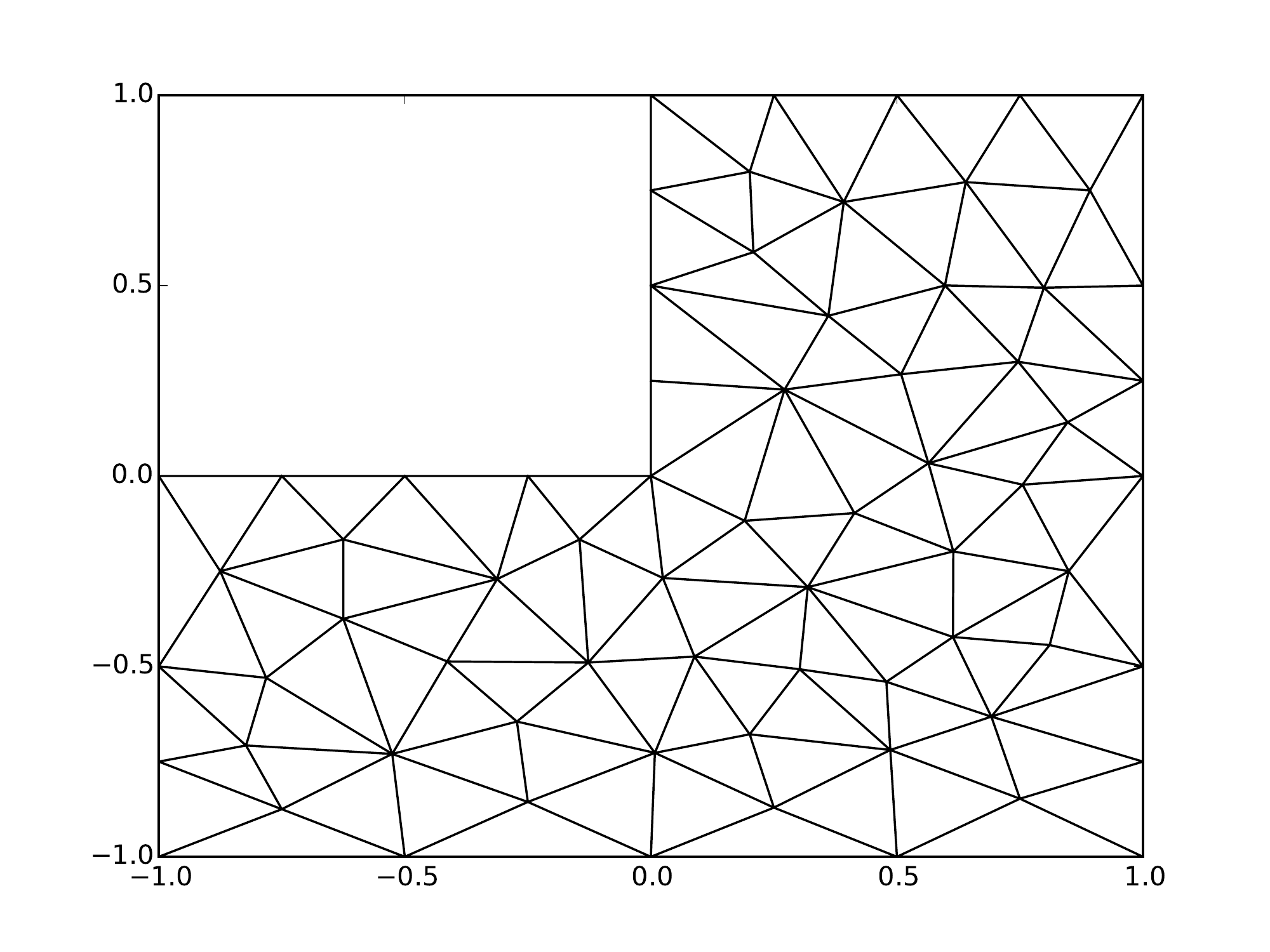}
    \includegraphics[width=0.49\textwidth]{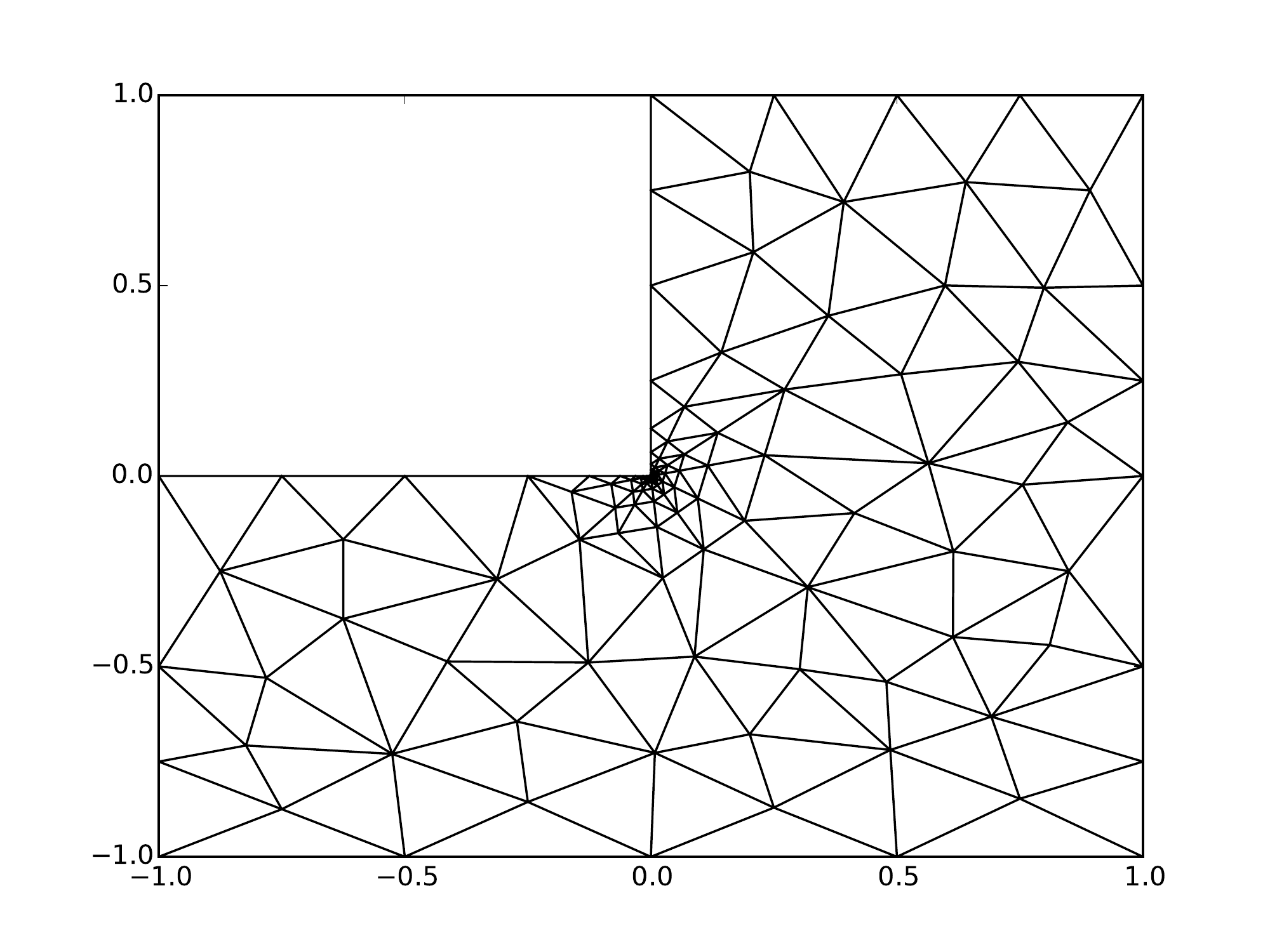}\\
    \includegraphics[width=0.49\textwidth]{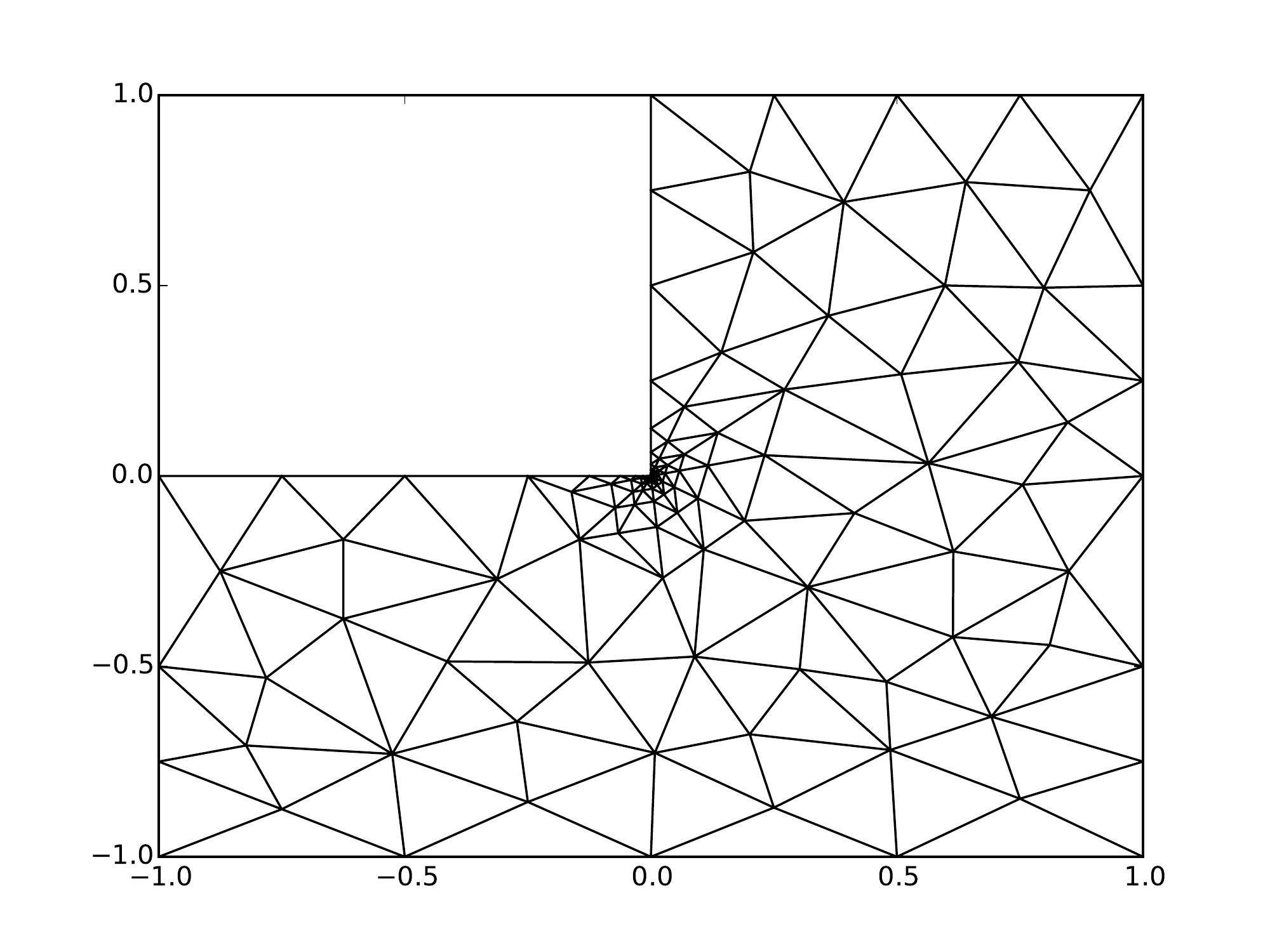}
    \includegraphics[width=0.49\textwidth]{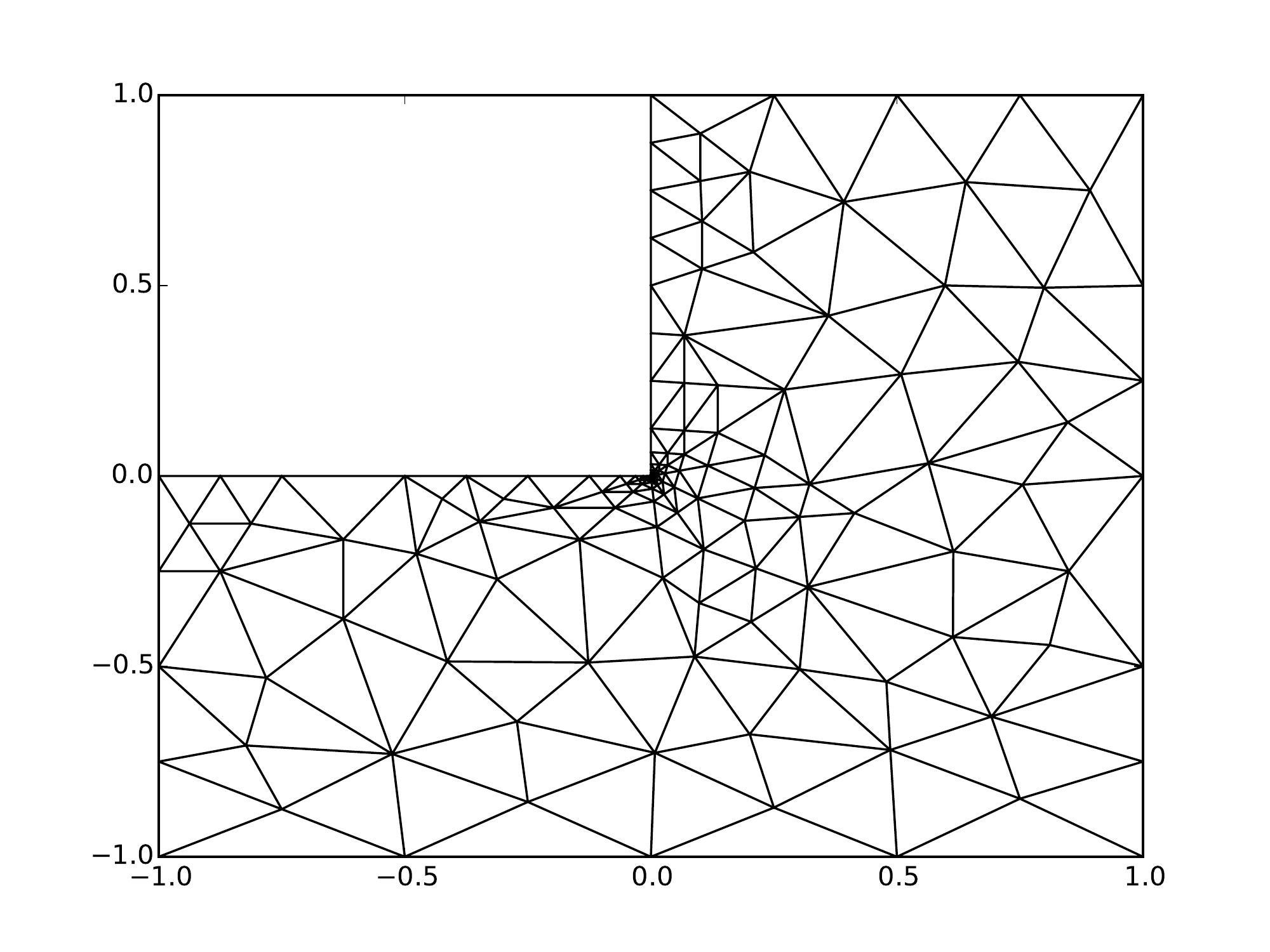}
    \caption{The initial (top right) and the final meshes with different boundary conditions.  The
        boundary conditions at the re-entrant corner are either simply supported (top
        right), clamped (bottom left) or free (bottom right). Surprisingly enough, the
        meshes for the simply supported and clamped boundaries end up being
    exactly the same.}
    \label{fig:mesh-lshaped}
\end{figure}

\pgfplotstableread{
    ndofs nelems eta
    593 111 2.02577144144
    2179 444 1.46810976277
    8348 1776 1.12381491681
    32674 7104 0.871033873118
}\uniformss

\pgfplotstableread{
    ndofs nelems eta
    593 111 2.02577144144
    638 121 1.86546429104
    727 140 1.59759130521
    824 160 1.23885967393
    921 180 0.969188164204
    1018 200 0.762907587751
    1115 220 0.603440209985
    1212 240 0.479294415861
}\adaptivess

\pgfplotstableread{
    ndofs nelems eta
    593 111 0.924193320451
    2179 444 0.595889771428
    8348 1776 0.40426945804
    32674 7104 0.275514833081
}\uniformcc

\pgfplotstableread{
    ndofs nelems eta
    593 111 0.924193320451
    638 121 0.725209767782
    727 140 0.579349194803
    824 160 0.392423628718
    921 180 0.269295860608
    1018 200 0.187398321262
    1115 220 0.133243366253
    1212 240 0.0982385833818
}\adaptivecc

\pgfplotstableread{
    ndofs nelems eta
    593 111 0.672667391309
    2179 444 0.388381670884
    8348 1776 0.229683630386
    32674 7104 0.137832309494
}\uniformff

\pgfplotstableread{
    ndofs nelems eta
    593 111 0.672667391309
    690 131 0.391628712819
    778 149 0.233378180614
    866 167 0.143258369597
    954 185 0.0917709133812
    1042 203 0.0633457066741
    1156 226 0.0454579296491
    1304 255 0.032930857058
}\adaptiveff

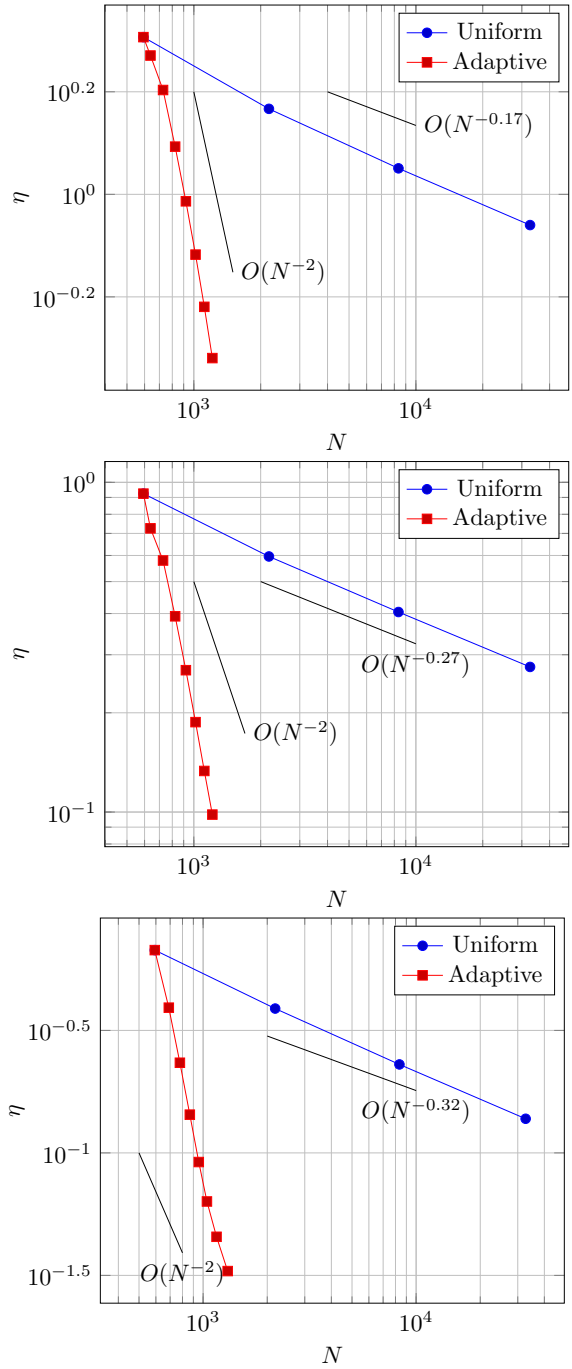
\begin{figure}
    \centering
    \begin{tikzpicture}[scale=0.9]
        \begin{axis}[
                xmode = log,
                ymode = log,
                xlabel = {$N$},
                ylabel = {$\eta$},
                grid = both
            ]
            \addplot table[x=ndofs,y=eta] {\uniformss};
            \addplot table[x=ndofs,y=eta] {\adaptivess};
            \addplot+ [black, domain=4e3:1e4, mark=none] {exp(-0.33/2*ln(x) + ln(1.585) - (-0.33/2)*ln(4e3)))} node[right,pos=1.0]{$O(N^{-0.17})$};
            \addplot+ [black, domain=1e3:15e2, mark=none] {exp(-4/2*ln(x) + ln(1.585) - (-4/2)*ln(1e3)))} node[right,pos=1.0]{$O(N^{-2})$};
            \addlegendentry{Uniform}
            \addlegendentry{Adaptive}
        \end{axis}
    \end{tikzpicture}
    \begin{tikzpicture}[scale=0.9]
        \begin{axis}[
                xmode = log,
                ymode = log,
                xlabel = {$N$},
                ylabel = {$\eta$},
                grid = both
            ]
            \addplot table[x=ndofs,y=eta] {\uniformcc};
            \addplot table[x=ndofs,y=eta] {\adaptivecc};
            \addplot+ [black, domain=2e3:1e4, mark=none] {exp(-0.54/2*ln(x) + ln(5e-1) - (-0.54/2)*ln(2e3)))} node[below,pos=1.0]{$O(N^{-0.27})$};
            \addplot+ [black, domain=1e3:1700, mark=none] {exp(-4/2*ln(x) + ln(5e-1) - (-4/2)*ln(1e3)))} node[right,pos=1.0]{$O(N^{-2})$};
            \addlegendentry{Uniform}
            \addlegendentry{Adaptive}
        \end{axis}
    \end{tikzpicture}
    \begin{tikzpicture}[scale=0.9]
        \begin{axis}[
                xmode = log,
                ymode = log,
                xlabel = {$N$},
                ylabel = {$\eta$},
                grid = both
            ]
            \addplot table[x=ndofs,y=eta] {\uniformff};
            \addplot table[x=ndofs,y=eta] {\adaptiveff};
            \addplot+ [black, domain=2e3:1e4, mark=none] {exp(-0.637/2*ln(x) + ln(3e-1) - (-0.637/2)*ln(2e3)))} node[below,pos=1.0]{$O(N^{-0.32})$};
            \addplot+ [black, domain=5e2:8e2, mark=none] {exp(-4/2*ln(x) + ln(1e-1) - (-4/2)*ln(5e2)))} node[below,pos=1.0]{$O(N^{-2})$};
            \addlegendentry{Uniform}
            \addlegendentry{Adaptive}
        \end{axis}
    \end{tikzpicture}
    \caption{L-shaped domain results. Simply supported (top), clamped (middle)
    and free (bottom) boundary conditions on the re-entrant corner.}
    \label{fig:lshapedgraph}
\end{figure}

%
%
%
%

\section*{Acknowledgements}
The authors thank the two anonymous referees and Prof. A. Ern  for comments that improved the final version of the paper.

\bibliography{plates,DG-AP}
\bibliographystyle{siamplain}
\end{document}